\documentclass{amsart}

\usepackage{amssymb}

\newtheorem{theorem}{Theorem}[section]
\newtheorem{lemma}[theorem]{Lemma}
\newtheorem{proposition}[theorem]{Proposition}
\newtheorem{corollary}[theorem]{Corollary}

\theoremstyle{remark}
\newtheorem{remark}[theorem]{Remark}

\begin{document}

\title{$\mathcal H$-Harmonic Bergman Projection on the Hyperbolic Ball}

\thanks{This research is supported by Eski\c{s}ehir Technical University
Research Fund under grant 20ADP085.}

\author{A. Ers$\dot{\hbox{\i}}$n \"Ureyen}
\address{Department of Mathematics, Faculty of Science,
Eki\c{s}ehir Technical University, 26470, Eski\c{s}ehir,
Turkey}
\email{aeureyen@eskisehir.edu.tr}

\date{\today}



\begin{abstract}
We determine precisely when the Bergman projection $P_\beta$
is bound\-ed from Lebesgue spaces $L^p_\alpha$ to weighted
Bergman spaces $\mathcal B^p_\alpha$ of $\mathcal H$-harmonic
functions on the hyperbolic ball,
and verify a recent conjecture of M. Stoll.
We obtain upper estimates for the reproducing kernel of the
$\mathcal H$-harmonic Bergman space $\mathcal B^2_\alpha$ and its
partial derivatives.
We also consider the projection from $L^\infty$ to the Bloch
space $\mathcal B$ of $\mathcal H$-harmonic functions.
\end{abstract}

\maketitle

\section{Introduction}\label{SI}

We follow the notation and terminology of the book \cite{St1} which we
refer the reader for more details.
For $n\geq 2$, let $\langle\cdot,\cdot\rangle$ be the Euclidean
inner product and $|\cdot|$ the corresponding norm in $\mathbb R^n$.
Let $\mathbb B=\{x\in\mathbb R^n\,:\,|x|<1\}$ be the unit ball and
$\mathbb S=\partial\mathbb B$ the unit sphere.

The hyperbolic ball is $\mathbb B$ equipped with the hyperbolic metric
$$ds=\frac{2}{1-|x|^2}\,|dx|.$$
For $f\in C^2(\mathbb B)$, the Laplace-Beltrami operator associated
with the hyperbolic metric (the invariant Laplacian) is given by
(up to a constant factor)
$$\Delta_hf(x)=(1-|x|^2)\bigl[(1-|x|^2)\Delta f(x)+2(n-2)\langle x,\nabla f(x)\rangle\bigr],$$
where $\Delta f=\tfrac{\partial^2f}{\partial x_1^2}+\dots+\tfrac{\partial^2f}{\partial x_n^2}$ and
$\nabla f=\bigl(\tfrac{\partial f}{\partial x_1},\dots,\tfrac{\partial f}{\partial x_n}\bigr)$
are the Euclidean Laplacian and gradient.
It is straightforward to show that $\Delta_hf(a)=\Delta(f\circ\varphi_a)(0)$,
where
$$\varphi_a(x)=\frac{a|x-a|^2+(1-|a|^2)(a-x)}{[x,a]^2},$$
is a M\"obius transformation mapping $\mathbb B$ to $\mathbb B$
and exchanging $a$ and $0$.
Here $[x,a]$ is defined as
$$[x,a]:=\sqrt{1-2\langle x,a\rangle+|x|^2|a|^2}\qquad (x,a\in\mathbb B).$$

A complex-valued function $f\in C^2(\mathbb B)$ is called hyperbolic harmonic
or $\mathcal H$-harmonic (sometimes also called $\mathcal M$-harmonic)
on $\mathbb B$, if $\Delta_hf(x)=0$ for all $x\in\mathbb B$.
We denote the space of all $\mathcal H$-harmonic functions by
$\mathcal H(\mathbb B)$.
In case $n=2$, $\mathcal H$-harmonic functions coincide with Euclidean
harmonic functions for which all the results mentioned below are well-known.
Therefore, from now on we assume $n\geq 3$.

Let $\nu$ be the Lebesgue measure on $\mathbb B$ normalized so that
$\nu(\mathbb B)=1$.
For $\alpha>-1$, define the weighted measures
$$d\nu_\alpha(x)=\frac{1}{V_\alpha}(1-|x|^2)^\alpha\,d\nu(x),$$
where the normalizing constant
\begin{equation}\label{Valpha}
V_\alpha=\frac{n}{2}B\Bigl(\frac{n}{2},\alpha+1\Bigr)
\end{equation}
so that $\nu_\alpha(\mathbb B)=1$.
For $0<p<\infty$ and $\alpha>-1$, the Lebesgue space
$L^p_\alpha=L^p(\mathbb B,d\nu_\alpha)$ consists of all measurable functions $f$
on $\mathbb B$ such that
$$\|f\|_{p,\alpha}=\biggl(\int_\mathbb B |f(x)|^p\,d\nu_\alpha(x)\biggr)^{1/p}<\infty.$$
The subspace $\mathcal B^p_\alpha$ consisting of $\mathcal H$-harmonic
functions is called the $\mathcal H$-harmonic weighted Bergman space,
$$\mathcal B^p_\alpha=\{f\in\mathcal H(\mathbb B)\,:\,\|f\|_{p,\alpha}<\infty\}.$$

When $1\leq p<\infty$, $\mathcal B^p_\alpha$ is a Banach space on which
point evaluation functionals $f \mapsto f(x)$ are bounded.
In particular, $\mathcal B^2_\alpha$ is a reproducing kernel Hilbert space
and for each $x\in\mathbb B$, there exists a unique
$\mathcal R_\alpha(x,\cdot)\in\mathcal B^2_\alpha$ such that
\begin{equation}\label{Reproduce}
f(x)=\int_\mathbb B f(y)\overline{\mathcal R_\alpha(x,y)}\,d\nu_\alpha(y)
\qquad (f\in\mathcal B^2_\alpha).
\end{equation}
The reproducing kernel $\mathcal R_\alpha(\cdot,\cdot)$ is real-valued
(so the conjugation in \eqref{Reproduce} can be removed),
symmetric in its variables and is $\mathcal H$-harmonic with respect to each variable.

For $\alpha>-1$, the Bergman projection operator $P_\alpha$ is defined by
$$P_\alpha f(x)=\int_\mathbb B \mathcal R_\alpha(x,y)f(y)\,d\nu_\alpha(y)$$
for suitable $f$.
It follows from \eqref{Reproduce} that $P_\alpha$ is the orthogonal projection of
$L^2_\alpha$ onto $\mathcal B^2_\alpha$.

The main purpose of this paper is to determine exactly when the operators
$P_\beta:L^p_\alpha\to\mathcal B^p_\alpha$ are bounded.

\begin{theorem}\label{Tprojection}
Let $1\leq p<\infty$ and $\alpha,\beta>-1$.
Then $P_\beta:L^p_\alpha\to\mathcal B^p_\alpha$ is bounded if and only if
\begin{equation}\label{Projcon}
 \alpha+1<p(\beta+1).
\end{equation}
In case \eqref{Projcon} holds, $P_\beta$ restricted to $\mathcal B^p_\alpha$
is identity.
\end{theorem}

This verifies a recent conjecture of M. Stoll (see \cite[Section 6]{St2}).
Actually, when the dimension $n$ is \textit{even}, the sufficiency part of
this theorem has already been proved in \cite[Theorem IV.14]{Pen}.
The difference between even and odd dimensions and the reason why the
problem is more difficult in odd dimensions will be explained below.
We remark that Theorem \ref{Tprojection} is true for all $n\geq 3$
regardless of the parity of the dimension and
shows that the condition \eqref{Projcon} is both necessary and sufficient.

To prove Theorem \ref{Tprojection} we need upper and lower estimates of
the kernel $\mathcal R_\alpha$.
This is a nontrivial problem since there is no known closed formula
for $\mathcal R_\alpha$.
However, it is known that it has the series expansion (\cite[Corollary III.5]{Pen},
\cite[Theorem 5.3]{St2})
\begin{equation}\label{RepKer}
\mathcal R_\alpha(x,y)=\sum_{m=0}^\infty c_m(\alpha)S_m(|x|)S_m(|y|)\,Z_m(x,y),
\end{equation}
where
\begin{equation}\label{DefineSmr}
S_m(r)
=\frac{F(m,1-\tfrac{n}{2};m+\tfrac{n}{2};r^2)}{F(m,1-\tfrac{n}{2};m+\tfrac{n}{2};1)},
\end{equation}
with $F={}_2F_1$ is the Gauss hypergeometric function, $Z_m$ is the zonal harmonic of
degree $m$ and the coefficients $c_m(\alpha)$ are given by
\begin{equation}\label{cm}
\frac{1}{c_m(\alpha)}=\frac{1}{V_\alpha}\,
n\int_0^1 r^{2m+n-1}(1-r^2)^\alpha\,S_m^2(r)\,dr=:I_m.
\end{equation}

We compare this with the Euclidean harmonic case.
Harmonic weighted Bergman space $b^2_\alpha$ has the reproducing kernel
(see \cite[p.~164]{DS}, \cite[Proposition 3]{Mia})
\begin{equation}\label{EucKer}
R_\alpha(x,y)=\sum_{m=0}^\infty\gamma_m(\alpha)\,Z_m(x,y),
\end{equation}
where
\begin{equation}\label{Euclgamma}
\frac{1}{\gamma_m(\alpha)}
=\frac{1}{V_\alpha}\,n\int_0^1 r^{2m+n-1}(1-r^2)^\alpha\,dr
=\frac{\Gamma(\alpha+\tfrac{n}{2}+1)\Gamma(m+\tfrac{n}{2})}
{\Gamma(\tfrac{n}{2})\Gamma(m+\alpha+\tfrac{n}{2}+1)}.
\end{equation}
Note that formulas may differ by a constant factor
depending on whether the normalizing constant $V_\alpha$ is used or not.
Upper and lower estimates of the Euclidean Bergman kernel $R_\alpha$ and its
derivatives are by now well-established.
If we compare \eqref{RepKer} and \eqref{EucKer}, we see that in \eqref{RepKer}
there appears the extra factor $S_m$.
However it is not difficult to deal with this extra term and
the main difficulty lies in the coefficient $c_m(\alpha)$.
Although the integral in \eqref{Euclgamma} is easily evaluated,
a closed formula for the hypergeometric integral in \eqref{cm} is not known.

We remark that in \textit{even} dimensions the hypergeometric function in
\eqref{DefineSmr} is just a polynomial since $1-\tfrac{n}{2}$ is a negative integer
and the hypergeometric series terminates.
For this reason it is easier to work in even dimensions and when the dimension is
also small the integral in \eqref{cm} can be explicitly evaluated and a closed
formula for $c_m(\alpha)$ can be written.
We will not use this fact and all our results and methods below will be independent
of dimension and will equally work in both even and odd dimensions.

It is shown in \cite[Theorem III.6]{Pen} that
\begin{equation}\label{gammamAsym}
c_m(\alpha)\sim m^{\alpha+1}\qquad (m\to\infty).
\end{equation}
We will show later that the following stronger estimate is true
\begin{equation}\label{Kis1}
c_m(\alpha)=\frac{\Gamma(m+\alpha+n)}{\Gamma(m+n-1)}
\biggl(D_0+O\Bigl(\frac{1}{m}\Bigr)\biggr)\qquad (m\to\infty),
\end{equation}
where the constant $D_0=D_{0}(\alpha,n)$ is given by
$$D_0
=\frac{\Gamma^2(n-1)\Gamma(\alpha+n+1)\Gamma(\tfrac{\alpha}{2}+1)\Gamma(\tfrac{\alpha}{2}+n)}
{\Gamma(\tfrac{n}{2})\Gamma(\alpha+\tfrac{n}{2}+1)\Gamma(\alpha+2n-1)
\Gamma(\tfrac{\alpha+n}{2})\Gamma(\tfrac{\alpha+n}{2}+1)}.$$
However, neither of these are sufficient to estimate $\mathcal R_\alpha$.
Our first aim is to show that $1/c_m(\alpha)=I_m$ is given by the following series.

\begin{theorem}\label{TIm}
There exist constants $A_k$ $(k=0,1,2,\dots)$
depending only on $\alpha$ and $n$ such that for $m\geq 1$,
\begin{equation}\label{ImSeries}
I_m=\frac{\Gamma(m+n-1)}{\Gamma(m+\alpha+n)}
\sum_{k=0}^\infty\frac{A_k}{(m+\alpha+n)_k}.
\end{equation}
The exact value of $A_k$ is given by
\begin{align}\label{Ak}
\begin{split}
A_k=&\frac{\Gamma(\tfrac{n}{2})\Gamma(\alpha+n)\Gamma(\alpha+2n-1)}
{\Gamma^2(n-1)\Gamma(\alpha+\tfrac{3n}{2})}\,
\frac{(\alpha+1)_k\,(\alpha+n)_k\,(1-\tfrac{n}{2})_k}
{(\alpha+\tfrac{n}{2}+1)_k\, k!}\\
&\times
{}_3F_2\left[
\begin{aligned}
& \tfrac{n}{2},\,\alpha+n,\,1-\tfrac{n}{2}\\
& \alpha+\tfrac{n}{2}+1+k,\,\alpha+\tfrac{3n}{2}
\end{aligned}\,;
1\right].
\end{split}
\end{align}
\end{theorem}

The formula for $A_k$ is complicated but these exact values will not be important
for the size estimates of $\mathcal R_\alpha$.
What matters only is the form of the series in \eqref{ImSeries} which, in fact,
is a convergent asymptotic series since $(m+\alpha+n)_k$ is a polynomial
in $m$ of degree $k$.
From this the asymptotic expansion of $c_m(\alpha)$ will easily follow
(see Corollary \ref{CAsymD} below) which will lead to the following upper estimate of
$\mathcal R_\alpha$.

\begin{theorem}\label{TEstBergman}
Let $\alpha>-1$.
There exists a constant $C>0$ depending only on $\alpha$ and $n$
such that for all $x,y\in\mathbb B$,
$$|\mathcal R_\alpha(x,y)|\leq \dfrac{C}{[x,y]^{\alpha+n}}.$$
Moreover, the exponent $(\alpha+n)$ is best possible.
\end{theorem}

Comparing this with the Euclidean case, we see that
exactly same estimate holds for $R_\alpha$.
This leads, naturally, to the fact that the condition \eqref{Projcon}
is same as the Euclidean case.
As a matter of fact this condition is also same
for  Bergman spaces of \textit{holomorphic} functions, see \cite[Theorem 1.10]{HKZ}.

By a common technique (applying Schur's test) the upper estimate in
Theorem \ref{TEstBergman} proves the ``if" part of Theorem \ref{Tprojection}.
However to show the ``only if" part we need lower estimates of $\mathcal R_\alpha$.
To achieve this we first estimate first-order partial derivatives of $\mathcal R_\alpha$.

\begin{theorem}\label{TPartial}
There exists a constant $C>0$ depending only on $\alpha$ and $n$ such that
for all $x,y\in\mathbb B$ and $i=1,2,\dots,n$,
$$\Bigl|\frac{\partial}{\partial x_i}\mathcal R_\alpha(x,y)\Bigr|
\leq \frac{C}{[x,y]^{\alpha+n+1}}.$$
\end{theorem}

We remark that this theorem is true for all $n\geq 3$.
On the other hand, presumably, higher-order derivatives
of $\mathcal R_\alpha$ will be dimension dependent and will behave
differently in even and odd dimensions, see \cite{GJ} and \cite{Ja2} for
similar results in other contexts.

The required lower estimate will be obtained in Proposition \ref{Plower} below which will
lead to the following two-sided estimate of weighted integrals of $\mathcal R_\alpha$
which is of independent interest.

\begin{theorem}\label{TIntegral}
For $\alpha,\beta>-1$ and $0<p<\infty$, set $c=p(\alpha+n)-(\beta+n)$.
For all $x\in\mathbb B$,
$$\int_\mathbb B\bigl|\mathcal R_\alpha(x,y)\bigr|^{p}(1-|y|^2)^\beta\,d\nu(y)
\sim\begin{cases}
\dfrac{1}{(1-|x|^2)^{c}},&\text{if $c>0$};\\
1+\log\dfrac{1}{1-|x|^2},&\text{if $c=0$};\\
1,&\text{if $c<0$}.
\end{cases}$$
The implied constants depend only on $a,b,\alpha,n$ and are independent of $x$.
\end{theorem}

The ``only if" part of Theorem \ref{Tprojection} will follow from this
theorem.

The projection theorem immediately leads to duality.
The proof of the following corollary is similar to the proof of
\cite[Theorem 1.16]{HKZ} and is omitted.

\begin{corollary}
Let $\alpha>-1$ and $1<p<\infty$.
The dual of $\mathcal B^p_\alpha$ can be identified with $\mathcal B^q_\alpha$
under the pairing
$$\langle f,g\rangle_\alpha=\int_\mathbb B f(x)\overline{g(x)}\,d\nu_\alpha(x),$$
where $1/p+1/q=1$.
\end{corollary}

Finally we consider the case $p=\infty$.
Denote the space of all essentially bounded functions on $\mathbb B$ by
$L^\infty=L^\infty(\mathbb B)$.
The gradient $\nabla^hf$ of the hyperbolic metric (the invariant gradient) is
given by
$$(\nabla^hf)(a)=-\nabla(f\circ\varphi_a)(0)=(1-|a|^2)\nabla f(a)\qquad (a\in\mathbb B),$$
for $f\in C^1(\mathbb B)$.
The $\mathcal H$-harmonic Bloch space $\mathcal B$ consists of all $f\in\mathcal H(\mathbb B)$
such that
$$\sup_{x\in\mathbb B}\bigl|\nabla^hf(x)\bigr|
=\sup_{x\in\mathbb B}(1-|x|^2)|\nabla f(x)|<\infty$$
with norm
$$\|f\|_{\mathcal B}=|f(0)|+\sup_{x\in\mathbb B}|\nabla^hf(x)|.$$

\begin{theorem}\label{TProjBloch}
For every $\alpha>-1$, $P_\alpha:L^\infty\to\mathcal B$ is bounded.
\end{theorem}

This paper is organized as follows.
In the Preliminaries section we recall some known facts about (generalized)
hypergeometric functions and a few integral estimates that will be used later.
In Section \ref{SA} we prove Theorem \ref{TIm} and obtain asymptotic expansion
of the coefficient $c_m(\alpha)$.
In Section \ref{SE} we obtain the upper estimate of $\mathcal R_\alpha$.
Our method will be general and will also give an upper estimate of the Hardy kernel
improving \cite[Theorem 3.2]{St3}.
In Section \ref{SPar} we estimate partial derivatives of $\mathcal R_\alpha$
and prove Theorem \ref{TPartial}, and in Section \ref{SInt} we prove the two-sided
estimate in Theorem \ref{TIntegral}.
We prove the projection theorems, Theorems \ref{Tprojection} and \ref{TProjBloch},
in the final section.

\section{Preliminaries}\label{SP}

\subsection{Notation}

We use the letter $C$ to denote positive constants whose exact value
may differ at each occurrence.
For two positive expressions $X$ and $Y$ we write $X\lesssim Y$ to mean
$X\leq CY$ for some $C>0$.
The constant $C$ may depend on the parameters $\alpha,\beta,n$ etc. that are
fixed beforehand, but will be independent of the variables $x,y\in\mathbb B$ etc.
If both $X\lesssim Y$ and $Y\lesssim X$, we write $X\sim Y$.

The Beta function is given by
\begin{equation}\label{BetaF}
B(a,b)=\int_0^1 t^{a-1}(1-t)^{b-1}\,dt
=\frac{\Gamma(a)\Gamma(b)}{\Gamma(a+b)} \qquad (Re\{a\}>0, Re\{b\}>0),
\end{equation}
where $\Gamma$ is the Euler Gamma function.

For $a\in\mathbb C$ and $k$ non-negative integer, the rising factorial
$(a)_k$ is defined as $(a)_0:=1$ and for $k\geq 1$, $(a)_k:=a(a+1)\dots(a+k-1)$.
If $a\neq 0,-1,-2,\dots$,
\begin{equation}\label{Poch}
(a)_k=\frac{\Gamma(a+k)}{\Gamma(a)}.
\end{equation}
By Stirling's formula, for $a,b\in\mathbb C$ and $k$ non-negative integer
\begin{equation}\label{Stirling}
\frac{\Gamma(k+a)}{\Gamma(k+b)}\sim k^{a-b},
\qquad \frac{(a)_k}{(b)_k}\sim k^{a-b},
\qquad k\to\infty.
\end{equation}

For $x,y\in\mathbb \mathbb R^n$, the Ahlfors bracket
$[x,y]=\sqrt{1-2\langle x,y\rangle+|x|^2|y|^2}$.
It is symmetric and if either of the variables is $0$, then $[x,0]=[0,y]=1$.
Otherwise
\begin{equation}\label{bracketxy}
[x,y]=\Bigl||y|x-\frac{y}{|y|}\Bigr|=\Bigl|\frac{x}{|x|}-|x|y\Bigr|.
\end{equation}
It is clear that
\begin{equation}\label{bracketEst}
[x,y]\geq 1-|x||y|.
\end{equation}

For $f\in L^1(\mathbb B)$, integration in polar coordinates is (see \cite[p.~6]{ABR})
\begin{equation}\label{Polar}
\int_\mathbb B f\,d\nu(x)
=n\int_0^1r^{n-1}\int_\mathbb S f(r\zeta)\,d\sigma(\zeta)dr,
\end{equation}
where $\sigma$ is normalized surface-area measure on $\mathbb S$.

\subsection{Zonal Harmonics}

We briefly review the basic properties of zonal harmonics that will be used later.
For details we refer the reader to \cite[Chapter 5]{ABR}.
Denote by $H_m(\mathbb R^n$), the linear space of all homogeneous harmonic
polynomials of degree $m$.
It is finite dimensional with $\text{dim}\, H_m\sim m^{n-2}$ as $m\to\infty$.
By homogeneity, $p_m\in H_m(\mathbb R^n)$ is determined by its restriction to $\mathbb S$.
This restriction is called a spherical harmonic and the space of all spherical harmonics
of degree $m$ is denoted $H_m(\mathbb S)$.
Spherical harmonics of different degrees are orthogonal in $L^2(\mathbb S)$:
\begin{equation}\label{OrthSpher}
\int_\mathbb S p_m(\zeta)\,p_k(\zeta)\,d\sigma(\zeta)=0
\qquad(m\neq k,\, p_m\in H_m(\mathbb S),\,p_k\in H_k(\mathbb S)).
\end{equation}
Point evaluation functionals are bounded on $H_m(\mathbb S)$.
So, for every $\eta\in\mathbb S$ there exists $Z_m(\cdot,\eta)\in H_m(\mathbb S)$,
called the zonal harmonic of degree $m$ with pole $\eta$, such that
\begin{equation}\label{ReproduceZm}
p_m(\eta)=\int_\mathbb S p_m(\zeta) Z_m(\zeta,\eta)\,d\sigma(\zeta)
\qquad(p_m\in H_m(\mathbb S)).
\end{equation}
$Z_m(\cdot,\cdot)$ is real-valued, symmetric, and homogenous of degree $m$ with respect
to each variable.
When $m=0$, $Z_0=1$ and when $m\geq 1$,
\begin{equation}\label{EstZm}
Z_m(\zeta,\zeta)=\text{dim}\,H_m\sim m^{n-2}
\ \text{ and }\
|Z_m(\zeta,\eta)|\leq Z_m(\zeta,\zeta)
\qquad (\zeta,\eta\in\mathbb S).
\end{equation}

\subsection{Hypergeometric Functions}

Let $a.b.c\in\mathbb C$ with $c\neq 0,-1,-2,\dots$.
For $|z|<1$, the Gauss hypergeometric function
$F(a,b;c;z)={}_2F_1(a,b;c;z)$ is defined by the series
\begin{equation}\label{F21Series}
F(a,b;c;z)=\sum_{k=0}^\infty\frac{(a)_k(b)_k}{(c)_k k!}\,z^k.
\end{equation}
The function $F(a,b;c;z)$ can be analytically continued to $\mathbb C\backslash[1,\infty)$
which we denote by the same symbol, see \cite[\S9.1]{Le}.

\begin{remark}\label{RHyp1}
If $Re\{c-a-b\}>0$, then by \eqref{Stirling}, the series in \eqref{F21Series}
converges absolutely and uniformly on the closed disk $|z|\leq 1$
and at $z=1$ by \cite[Eq.~9.3.4]{Le},
\begin{equation}\label{Gaussf}
F(a,b;c;1)=\frac{\Gamma(c)\Gamma(c-a-b)}{\Gamma(c-a)\Gamma(c-b)}.
\end{equation}
\end{remark}

For a proof of the following transformation formula, see \cite[Eq.~9.5.3]{Le}.

\begin{lemma}\label{LEulerIdty}
Let $a,b,c\in\mathbb C$ with $c\neq 0,-1,-2,\dots$.
For $z\in\mathbb C\backslash[1,\infty)$,
$$F(a,b;c;;z)=(1-z)^{c-a-b}F(c-a,c-b;c;z).$$
\end{lemma}

A proof of the following integral formula for $F(a,b;c;z)$ can be
found in \cite[Eq.~9.1.4]{Le}.
The functions in the integrand take their principal values.

\begin{lemma}\label{LF21Integral}
Suppose $Re\{c\}>Re\{a\}>0$.
Then for $|z|<1$,
$$F(a,b;c;;z)=\frac{\Gamma(c)}{\Gamma(a)\Gamma(c-a)}
\int_0^1 t^{a-1}(1-t)^{c-a-1}(1-zt)^{-b}\,dt.$$
Moreover, the right-hand side gives the analytic continuation of $F(a,b;c;z)$ to
$\mathbb C\backslash[1,\infty)$.
Therefore the above formula is true for all $z\in\mathbb C\backslash[1,\infty)$
\end{lemma}

\begin{remark}\label{RF21Pos}
The integral above shows that if $a,b,c$ are real and $c>a>0$, then $F(a,b;c;z)>0$ when
$-\infty<z<1$.
If, in addition $b<0$, then $F(a,b;c;z)$ is decreasing on the interval $-\infty<z<1$.
\end{remark}

The next identity appears in \cite[Eq.(4), p.399]{EMOT}.
For a proof, see \cite[Lemma 2.1]{Li}.

\begin{lemma}\label{LHyperInt}
Suppose $Re\{c\}>0$, $Re\{\rho\}>0$ and $Re\{c-a-b+\rho\}>0$.
Then
$$\int_0^1 t^{c-1}(1-t)^{\rho-1} F(a,b;c;t)\,dt
=\frac{\Gamma(c)\Gamma(\rho)\Gamma(c-a-b+\rho)}{\Gamma(c-a+\rho)\Gamma(c-b+\rho)}.$$
\end{lemma}

\subsection{Generalized Hypergeometric Function ${}_3F_2$}

Let $a,b,c,d,e\in\mathbb C$ with $d$ and $e$ are different than $0,-1,-2,\dots$.
For $|z|<1$, the generalized hypergeometric series
${}_3F_2\left[
\begin{aligned}
& a,\,b,\,c\\
& d,\,e
\end{aligned};
z\right]$
is defined by
\begin{equation}\label{DefF32}
{}_3F_2\left[
\begin{aligned}
& a,\,b,\,c\\
& d,\,e
\end{aligned}\,;z\right]
=\sum_{k=0}^\infty\frac{(a)_k(b)_k(c)_k}{(d)_k(e)_k\,k!}\,z^k.
\end{equation}
The function ${}_3F_2$ can be analytically continued to $\mathbb C\backslash[1,\infty)$
which we denote by the same symbol.

\begin{remark}\label{RF32at1}
If $Re\{d+e-a-b-c\}>0$, then by \eqref{Stirling} the series in \eqref{DefF32} converges absolutely
and uniformly for $|z|\leq 1$.
\end{remark}

\begin{remark}\label{RChangeOrder}
The ${}_{3}F_{2}$ is symmetric with respect to its upper variables $a,b,c$ and
they can be permuted.
Same is true for the lower variables $d,e$.
\end{remark}

Similar to Lemma \ref{LF21Integral}, there is an integral formula for ${}_3F_2$.
It can be verified by expanding $(1-zt_1t_2)^{-c}$ in series (see \cite[Eq.(1.5)]{E}).

\begin{lemma}\label{LF32Integral}
Suppose $Re\{d\}>Re\{a\}>0$ and $Re\{e\}>Re\{b\}>0$.
For $|z|<1$,
\begin{align*}
{}_{3}F_{2}\left[
\begin{aligned}
& a,\,b,\,c\\
& d,\,e
\end{aligned};
z\right]
=&\,\frac{\Gamma(d)\Gamma(e)}{\Gamma(a)\Gamma(d-a)\Gamma(b)\Gamma(e-b)}\\
&\times \int_0^1\int_0^1 t_1^{a-1}(1-t_1)^{d-a-1}t_2^{b-1}(1-t_2)^{e-b-1}(1-zt_1t_2)^{-c}\,dt_1t_2.
\end{align*}
Moreover, the right-hand side gives the analytic continuation of ${}_3F_2$ to
$\mathbb C\backslash [1,\infty)$.
\end{lemma}

\begin{lemma}\label{LF32Prop}
Suppose $a,b,c,d,e$ are all real with $d,e\neq 0,-1,-2,\dots$.
\begin{enumerate}
\item[(i)] If $d>a>0$ and $e>b>0$, then
${}_3F_2\left[
\begin{aligned}
& a,\,b,\,c\\
& d,\,e
\end{aligned}\,;z\right]>0$ for $-\infty<z<1$.
\item[(ii)] If $d>a>0$, $e>b>0$ and $c<0$, then
${}_3F_2\left[
\begin{aligned}
& a,\,b,\,c\\
& d,\,e
\end{aligned}\,;z\right]$  is decreasing on the interval $(-\infty,1)$.
\item[(iii)] If $d>a>0$, $e>b>0$ and $c<0$, then
$0\leq {}_3F_2\left[
\begin{aligned}
& a,\,b,\,c\\
& d,\,e
\end{aligned}\,;1\right]\leq 1$.
\end{enumerate}
\end{lemma}

\begin{proof}
Parts (i) and (ii) follow from the integral in Lemma \ref{LF32Integral}.
For part (iii) note that by Remark \ref{RF32at1}, ${}_3F_2\left[
\begin{aligned}
& a,\,b,\,c\\
& d,\,e
\end{aligned}\,;z\right]$ is well-defined and left continuous at $z=1$.
Since at $z=0$, ${}_3F_2\left[
\begin{aligned}
& a,\,b,\,c\\
& d,\,e
\end{aligned}\,;0\right]=1$, the result follows from parts (i) and (ii).
\end{proof}

The next identity is taken from \cite[Eq.(6), p.~399]{EMOT}.
However, there is a typo in \cite{EMOT} where $(1-z)^{-\sigma}$ is written as $(1-z)^\sigma$.
Therefore we give a proof.

\begin{lemma}\label{LIntF21}
Suppose $Re\{c\}>0$, $Re\{\rho\}>0$ and $Re\{c-a-b+\rho\}>0$.
Then for $z\in\mathbb C\backslash[1,\infty)$,
\begin{align*}
\int_0^1 t^{c-1}(1-t)^{\rho-1}(1-zt)^{-\sigma}F(a,b;c;t)\,dt
=&\frac{\Gamma(c)\Gamma(\rho)\Gamma(c-a-b+\rho)}{\Gamma(c-a+\rho)\Gamma(c-b+\rho)}
\ (1-z)^{-\sigma}\\
&\times {}_{3}F_{2}\left[
\begin{aligned}
& \rho,\,\sigma,\,c-a-b+\rho\,\\
& c-a+\rho,\,c-b+\rho
\end{aligned};
\frac{z}{z-1} \right].
\end{align*}
\end{lemma}

\begin{proof}
Call the integral on the left as $I$.
Assume first that $\big|\frac{z}{z-1}\big|<1$.
Because
$$1-zt=1-z+z(1-t)=(1-z)\Bigl(1-\frac{z}{z-1}(1-t)\Bigr),$$
we have
$$(1-zt)^{-\sigma}
=(1-z)^{-\sigma}\sum_{m=0}^\infty\frac{(\sigma)_k}{k!}\Bigl(\frac{z}{z-1}\Bigr)^k(1-t)^k,$$
where the series uniformly converges for $0\leq t\leq 1$.
Inserting this into the integral and changing the orders of the series and the integral we obtain
$$I=(1-z)^{-\sigma}\sum_{m=0}^\infty\frac{(\sigma)_k}{k!}\Big(\frac{z}{z-1}\Big)^k
\int_0^1 t^{c-1}(1-t)^{\rho+k-1}F(a,b;c;t)\,dt.$$
Computing the integral on the right with Lemma \ref{LHyperInt} and applying \eqref{Poch} proves the
result when $\big|\frac{z}{z-1}\big|<1$.
The general case follows from analytic continuation.
\end{proof}

We will need two identities about the value of an ${}_3F_2$ at $z=1$.
The first one is due to Kummer (see \cite[16.4.11]{OLBC} or
\cite[Corollary 3.3.5, p.~142]{AAR}).

\begin{lemma}\label{LKummer}
If $Re\{d+e-a-b-c\}>0$ and $Re\{e-a\}>0$, then
$${}_3F_2\left[
\begin{aligned}
& a,\,b,\,c\\
& d,\,e
\end{aligned}
\,;1\right]
=\frac{\Gamma(e)\Gamma(d+e-a-b-c)}{\Gamma(e-a)\Gamma(d+e-b-c)}\,
{}_3F_2\left[
\begin{aligned}
& a,\,d-b,\,d-c\\
& d,\,d+e-b-c
\end{aligned}
\,;1\right].$$
\end{lemma}

The next identity is called Dixon's well-poised sum
(see \cite[16.4.4]{OLBC} or \cite[Theorem 3.4.1, p.~143]{AAR}).
\begin{lemma}\label{LDixon}
If $Re\{a-2b-2c\}>-2$, then
$${}_3F_2\left[
\begin{aligned}
& a,\,b,\,c\\
& a-b+1,\,a-c+1
\end{aligned}
\,;1\right]
=\frac{\Gamma(\tfrac{a}{2}+1)\Gamma(a-b+1)\Gamma(a-c+1)\Gamma(\tfrac{a}{2}-b-c+1)}
{\Gamma(a+1)\Gamma(\tfrac{a}{2}-b+1)\Gamma(\tfrac{a}{2}-c+1)\Gamma(a-b-c+1)}.$$
\end{lemma}

\subsection{The Factor $S_m(r)$}

Let $p_m\in H_m(\mathbb S)$.
The solution of the Dirichlet problem for $p_m$ in the hyperbolic ball
is given by (see \cite[Section 6.1]{St1})
$$f(x)=S_m(|x|)\,p_m(x)
\qquad (x\in \overline{\mathbb B}).$$
That is, $f$ is $\mathcal H$-harmonic on $\mathbb B$, continuous
on $\overline{\mathbb B}$ and equals $p_m$ on $\mathbb S$.
Here, the factor $S_m(r)$ is as given in \eqref{DefineSmr}.
It is a hypergeometric function normalized so that $S_m(1)=1$.
When $m=0$, $S_0(r)=1$.
When $m\geq 1$ by \eqref{Gaussf},
$$F(m,1-\tfrac{n}{2};m+\tfrac{n}{2};1)
=\frac{\Gamma(m+\tfrac{n}{2})\Gamma(n-1)}{\Gamma(\tfrac{n}{2})\Gamma(m+n-1)}
=\frac{B(m,n-1)}{B(m,\frac{n}{2})}$$
so that
\begin{equation}\label{Smr}
S_m(r)=\frac{B(m,\frac{n}{2})}{B(m,n-1)}\,F(m,1-\tfrac{n}{2};m+\tfrac{n}{2};r^2).
\end{equation}

The following estimate is proved in \cite[Lemma 2.6]{St3}.
A different proof is given below.

\begin{lemma}\label{LSmrEst}
There exists a constant $C>0$ depending only on $n$ such that
for all $m\geq 1$ and $0\leq r\leq 1$,
$$1\leq S_m(r)\leq C m^{\tfrac{n}{2}-1}.$$
\end{lemma}

\begin{proof}
By \eqref{Smr} and Lemma \ref{LF21Integral}, the integral formula
\begin{equation}\label{IntSmr}
S_m(r)=\frac{1}{B(m,n-1)}\int_0^1 t^{m-1}(1-t)^{\frac{n}{2}-1}(1-r^2t)^{\frac{n}{2}-1}\, dt
\end{equation}
holds.
As in Remark \ref{RF21Pos}, it follows from the above integral that
$S_m(r)$ is decreasing on the interval $0\leq r\leq 1$.
Thus $S_m(1)\leq S_m(r)\leq S_m(0)$.
Putting $r=0$ in \eqref{Smr} shows $S_m(0)=B(m,\tfrac{n}{2})/B(m,n-1)$.
Thus
$$1\leq S_m(r)\leq \frac{B(m,\tfrac{n}{2})}{B(m,n-1)}
=\frac{\Gamma(\tfrac{n}{2})\Gamma(m+n-1)}{\Gamma(m+\tfrac{n}{2})\Gamma(n-1)},$$
and the result follows from \eqref{Stirling}.
\end{proof}

\subsection{Integral Estimates}
In the sequel we will need three integral estimates
which have been repeatedly proved in various places
and we don't know the original sources.
The first one is elementary.

\begin{lemma}\label{LInt01r}
Let $b>-1$ and $c\in\mathbb R$.
For $0\leq r<1$,
$$\int_0^1\frac{(1-t)^b\,dt}{(1-rt)^{1+b+c}}
\sim\begin{cases}
\dfrac{1}{(1-r)^{c}},&\text{if $c>0$};\\
1+\log\dfrac{1}{1-r},&\text{if $c=0$};\\
1,&\text{if $c<0$}.
\end{cases}$$
The implied constants depend only on $b,c$ and are independent of $r$.
\end{lemma}

A proof of the next lemma can be found in, for example, \cite[Lemma 6.1]{GKU}.

\begin{lemma}\label{LInt01}
Let $a,b>-1$ and $c\in\mathbb R$.
For $x\in\mathbb B$ and $\eta\in\mathbb S$,
$$\int_0^1\frac{t^{a}(1-t)^b\,dt}{|tx-\eta|^{1+b+c}}
\sim\begin{cases}
\dfrac{1}{|x-\eta|^{c}},&\text{if $c>0$};\\
1+\log\dfrac{1}{|x-\eta|},&\text{if $c=0$};\\
1,&\text{if $c<0$}.
\end{cases}$$
The implied constants depend only on $a,b,c,n$ and are independent of $x$.
\end{lemma}

The following lemma is the real analog of \cite[Proposition 1.4.10]{Ru}.
For a proof see, for example, \cite[Proposition 2.2]{LS}.

\begin{lemma}\label{LIntBracket}
Let $b>-1$ and $c\in\mathbb R$.
For $x\in\mathbb B$,
$$\int_\mathbb B\frac{1-|y|^2)^b\,d\nu(y)}{[x,y]^{n+b+c}}
\sim\begin{cases}
\dfrac{1}{(1-|x|^2)^{c}},&\text{if $c>0$};\\
1+\log\dfrac{1}{1-|x|^2},&\text{if $c=0$};\\
1,&\text{if $c<0$}.
\end{cases}$$
The implied constants depend only on $b,c,n$ and are independent of $x$.
\end{lemma}

\section{Asymptotic Expansion of the Coefficients of Reproducing Kernels}\label{SA}

The purpose of this section is to obtain an asymptotic expansion of the
coefficient $c_m(\alpha)$ of the reproducing kernel $\mathcal R_\alpha$.
For this we first prove Theorem \ref{TIm} and write $I_m=1/c_m(\alpha)$
as a series.
By \eqref{cm} and \eqref{Valpha}, for all $m\geq 0$,
\begin{equation}\label{ImInt1}
I_m=\frac{1}{B(\frac{n}{2},\alpha+1)}
\int_0^1 r^{m+\frac{n}{2}-1}(1-r)^\alpha\,S_m^2(\sqrt{r})\,dr.
\end{equation}
If $m=0$, then $S_0(r)=1$ and so $I_0$ and $c_0(\alpha)$ equal to $1$.
From now on, we assume $m\geq 1$.

\begin{proof}[Proof of Theorem \ref{TIm}]

Inserting the formula of $S_m(r)$ in \eqref{Smr} into \eqref{ImInt1} we obtain
\begin{equation}\label{Im}
I_m=\frac{B^2(m,\frac{n}{2})}{B(\frac{n}{2},\alpha+1)B^2(m,n-1)}
\int_0^1 r^{m+\frac{n}{2}-1}(1-r)^\alpha\,F^2(m,1-\tfrac{n}{2};m+\tfrac{n}{2};r)\,dr.
\end{equation}
We do not know a simple closed formula for this integral
and we look for an asymptotic expansion.
However in the above integral there are too many terms that depend on $m$ and
our first aim is to manipulate this integral by using the formulas listed
in Section \ref{SP} in such a way that the appearances of $m$
will be reduced.

First, note that by Lemmas \ref{LEulerIdty} and \ref{LF21Integral},
\begin{align*}
F(m,1-\tfrac{n}{2};m+\tfrac{n}{2};r)
&=(1-r)^{n-1}F(\tfrac{n}{2},m+n-1;m+\tfrac{n}{2};r)\\
&=\frac{(1-r)^{n-1}}{B(m,\tfrac{n}{2})}
\int_0^1 t^{\frac{n}{2}-1}(1-t)^{m-1}(1-rt)^{-(m+n-1)}\,dt.
\end{align*}
Inserting this into \eqref{Im} for one of the $F(m,1-\tfrac{n}{2};m+\tfrac{n}{2};r)$
and changing the orders of the integrals (possible by Tonelli's theorem
since the integrand is non-negative by Remark \ref{RF21Pos}), we obtain
\begin{multline*}
I_m=\frac{B(m,\frac{n}{2})}{B(\frac{n}{2},\alpha+1)B^2(m,n-1)}
\int_0^1 t^{\frac{n}{2}-1}(1-t)^{m-1}
\int_0^1
r^{m+\frac{n}{2}-1}(1-r)^{\alpha+n-1}\\
\times (1-rt)^{-(m+n-1)}
F(m,1-\tfrac{n}{2};m+\tfrac{n}{2};r)\,dr dt.
\end{multline*}
Computing the inner integral with Lemma \ref{LIntF21} gives
(with also Remark \ref{RChangeOrder})
\begin{align*}
I_m
=&\frac{B(m,\frac{n}{2})}{B(\frac{n}{2},\alpha+1)B^2(m,n-1)}\,
\frac{\Gamma(m+\frac{n}{2})\Gamma(\alpha+n)\Gamma(\alpha+2n-1)}
{\Gamma(\alpha+\frac{3n}{2})\Gamma(m+\alpha+2n-1)}\\
&\times \int_0^1 t^{\frac{n}{2}-1}(1-t)^{-n}\,
{}_{3}F_{2}\left[
\begin{aligned}
& m+n-1,\,\alpha+n,\,\alpha+2n-1\,\\
& m+\alpha+2n-1,\,\alpha+\tfrac{3n}{2}
\end{aligned};\,\frac{t}{t-1} \right]\,dt.
\end{align*}
Although in the integrand there is the term $(1-t)^{-n}$, the above integral
is convergent.
This is because the ${}_3F_2$ term is positive by Lemma \ref{LF32Prop} (i) and
is dominated by
$(1-t)^{\min\{m+n-1,\,\alpha+n\}}\log^2\bigl(\frac{1}{1-t}\bigr)$ as $t\to 1^-$
(see \cite[p.~570]{BMS}).

We next apply Lemma \ref{LF32Integral} and write the ${}_3F_2$ in the above
integral as a double integral.
After the cancellations we obtain
\begin{align*}
I_m
&=\frac{B(m,\frac{n}{2})}{B(\frac{n}{2},\alpha+1)B^2(m,n-1)}\,
\frac{\Gamma(m+\frac{n}{2})\Gamma(\alpha+2n-1)}
{\Gamma(m+n-1)\Gamma(\alpha+n)\Gamma(\frac{n}{2})}
\int_0^1 t^{\frac{n}{2}-1}(1-t)^{-n}\\
&\times \int_0^1\int_0^1
u^{m+n-2}(1-u)^{\alpha+n-1} v^{\alpha+n-1}(1-v)^{\frac{n}{2}-1}
\bigl(1-uv\tfrac{t}{t-1}\bigr)^{-(\alpha+2n-1)}\,dudvdt.
\end{align*}
Observe that in the integrand there is only one term that depends on $m$
and we can now obtain the asymptotic expansion of $I_m$.
Write
$$1-uv\frac{t}{t-1}=(1-t)^{-1}(1-t+uvt)=(1-t)^{-1}(1-t(1-uv)),$$
and change the orders of the integrals which is possible
since every term in the integrand is non-negative.
Replace also the Beta functions with Gamma functions as in \eqref{BetaF} and simplify.
This shows
\begin{align*}
I_m
=&\frac{\Gamma(\alpha+\tfrac{n}{2}+1)\Gamma(\alpha+2n-1)}
{\Gamma(\tfrac{n}{2})\Gamma^2(n-1)\Gamma(\alpha+1)\Gamma(\alpha+n)}\,
\frac{\Gamma(m+n-1)}{\Gamma(m)}
\int_0^1\int_0^1
u^{m+n-2}(1-u)^{\alpha+n-1}\\
&\times v^{\alpha+n-1}(1-v)^{\frac{n}{2}-1}
\int_0^1 t^{\frac{n}{2}-1}(1-t)^{\alpha+n-1}\bigl(1-t(1-uv)\bigr)^{-(\alpha+2n-1)}\,dtdudv.
\end{align*}
The inner integral is a hypergeometric function by Lemma \ref{LF21Integral}:
\begin{align*}
\int_0^1 t^{\frac{n}{2}-1}(1-t)^{\alpha+n-1}&\bigl(1-t(1-uv)\bigr)^{-(\alpha+2n-1)}\,dt\\
&=\frac{\Gamma(\tfrac{n}{2})\Gamma(\alpha+n)}{\Gamma(\alpha+\tfrac{3n}{2})}\,
F(\tfrac{n}{2},\alpha+2n-1;\alpha+\tfrac{3n}{2};1-uv)\\
&=\frac{\Gamma(\tfrac{n}{2})\Gamma(\alpha+n)}{\Gamma(\alpha+\tfrac{3n}{2})}\,
(uv)^{-n+1}\,F(\alpha+n,1-\tfrac{n}{2};\alpha+\tfrac{3n}{2};1-uv),
\end{align*}
where in the last equality we apply Lemma \ref{LEulerIdty}.
The reason for this last step is to get the parameter $1-\tfrac{n}{2}$ in the
hypergeometric function so that when $n$ is even the sum will terminate.
Thus
\begin{align*}
I_m
&=\frac{\Gamma(\alpha+\tfrac{n}{2}+1)\Gamma(\alpha+2n-1)}
{\Gamma^2(n-1)\Gamma(\alpha+1)\Gamma(\alpha+\tfrac{3n}{2})}\,
\frac{\Gamma(m+n-1)}{\Gamma(m)}\\
&\times\int_0^1\int_0^1
u^{m-1}(1-u)^{\alpha+n-1}\, v^\alpha(1-v)^{\frac{n}{2}-1}
F(\alpha+n,1-\tfrac{n}{2};\alpha+\tfrac{3n}{2};1-uv)\,dudv.
\end{align*}
We next expand the above hypergeometric function in series.
By Remark \ref{RHyp1}, this expansion converges uniformly for $0\leq u,v\leq 1$ and
changing the orders of the series and the integrals shows
\begin{align*}
I_m
&=\frac{\Gamma(\alpha+\tfrac{n}{2}+1)\Gamma(\alpha+2n-1)}
{\Gamma^2(n-1)\Gamma(\alpha+1)\Gamma(\alpha+\tfrac{3n}{2})}\,
\frac{\Gamma(m+n-1)}{\Gamma(m)}\\
&\times \sum_{j=0}^\infty\frac{(\alpha+n)_j(1-\tfrac{n}{2})_j}{(\alpha+\tfrac{3n}{2})_j\,j!}
\int_0^1\int_0^1 u^{m-1}(1-u)^{\alpha+n-1}\, v^\alpha(1-v)^{\frac{n}{2}-1}
(1-uv)^j\,dudv.
\end{align*}
Since $1-uv=(1-v)+v(1-u)$, we have
$(1-uv)^j=\sum_{k=0}^j \binom{j}{k}v^k(1-u)^k(1-v)^{j-k}.$
Inserting this into the above integral gives
\begin{align*}
I_m
=&\frac{\Gamma(\alpha+\tfrac{n}{2}+1)\Gamma(\alpha+2n-1)}
{\Gamma^2(n-1)\Gamma(\alpha+1)\Gamma(\alpha+\tfrac{3n}{2})}\,
\frac{\Gamma(m+n-1)}{\Gamma(m)}\,
\sum_{j=0}^\infty\frac{(\alpha+n)_j(1-\tfrac{n}{2})_j}{(\alpha+\tfrac{3n}{2})_j\, j!}\\
&\times\sum_{k=0}^j \binom{j}{k}
\int_0^1\int_0^1 u^{m-1}(1-u)^{\alpha+n-1+k}\, v^{\alpha+k}(1-v)^{\frac{n}{2}-1+j-k}\,dudv.
\end{align*}
We compute the integrals with \eqref{BetaF} and after the cancellation obtain
\begin{align*}
I_m
=&\frac{\Gamma(\alpha+\tfrac{n}{2}+1)\Gamma(\alpha+2n-1)}
{\Gamma^2(n-1)\Gamma(\alpha+1)\Gamma(\alpha+\tfrac{3n}{2})}\,\Gamma(m+n-1)\\
&\times
\sum_{j=0}^\infty\frac{(\alpha+n)_j(1-\tfrac{n}{2})_j}{(\alpha+\tfrac{3n}{2})_j\, j!}\,
\sum_{k=0}^j \binom{j}{k}
\frac{\Gamma(\alpha+n+k)\Gamma(\alpha+1+k)\Gamma(\tfrac{n}{2}+j-k)}
{\Gamma(m+\alpha+n+k)\Gamma(\alpha+\tfrac{n}{2}+1+j)}.
\end{align*}
We change the orders of the sums by using the identity (see \cite[Lemma 10, p.~56]{Ra})
$$\sum_{j=0}^\infty\sum_{k=0}^j f(j,k)
=\sum_{k=0}^\infty\sum_{j=k}^\infty f(j,k)
=\sum_{k=0}^\infty\sum_{j=0}^\infty f(j+k,k).$$
This is possible because except $(1-\tfrac{n}{2})_j$, all the terms in the
sum are positive.
Also if $n$ is even, then  $(1-\tfrac{n}{2})_j=0$ when $j\geq\tfrac{n}{2}$ and
the sum is finite.
If $n$ is odd, then  $(1-\tfrac{n}{2})_j$ has the same sign for every
$j\geq\tfrac{n-1}{2}$ and except for finitely many terms the
summands are either all positive or all negative.
Thus
\begin{align*}
I_m
=&\frac{\Gamma(\alpha+\tfrac{n}{2}+1)\Gamma(\alpha+2n-1)}
{\Gamma^2(n-1)\Gamma(\alpha+1)\Gamma(\alpha+\tfrac{3n}{2})}\,\Gamma(m+n-1)\,
\sum_{k=0}^\infty\frac{\Gamma(\alpha+n+k)\Gamma(\alpha+1+k)}{\Gamma(m+\alpha+n+k)\, k!}\\
&\times
\sum_{j=0}^\infty\frac{(\alpha+n)_{j+k}(1-\tfrac{n}{2})_{j+k}}{(\alpha+\tfrac{3n}{2})_{j+k}\, j!}\,
\frac{\Gamma(\tfrac{n}{2}+j)}{\Gamma(\alpha+\tfrac{n}{2}+1+j+k)}
\end{align*}
By \eqref{Poch}, $\Gamma(\alpha+n+k)=\Gamma(\alpha+n)(\alpha+n)_k$
and a similar identity holds for the other Gamma functions in the summand.
Writing these and doing the cancellations shows
\begin{align*}
I_m
=&\frac{\Gamma(\alpha+2n-1)\Gamma(\alpha+n)\Gamma(\tfrac{n}{2})}
{\Gamma^2(n-1)\Gamma(\alpha+\tfrac{3n}{2})}\,
\frac{\Gamma(m+n-1)}{\Gamma(m+\alpha+n)}
\sum_{k=0}^\infty\frac{(\alpha+n)_k\,(\alpha+1)_k}{(m+\alpha+n)_k\, k!}\\
&\times
\sum_{j=0}^\infty\frac{(\alpha+n)_{j+k}\,(1-\tfrac{n}{2})_{j+k}\,(\tfrac{n}{2})_j}
{(\alpha+\tfrac{3n}{2})_{j+k}\,(\alpha+\tfrac{n}{2}+1)_{j+k}\, j!}.
\end{align*}
We next use the elementary identity $(a)_{j+k}=(a)_k(a+k)_j$ and obtain
\begin{align*}
&{}I_m=
\frac{\Gamma(\alpha+2n-1)\Gamma(\alpha+n)\Gamma(\tfrac{n}{2})}
{\Gamma^2(n-1)\Gamma(\alpha+\tfrac{3n}{2})}\,
\frac{\Gamma(m+n-1)}{\Gamma(m+\alpha+n)}\\
&\sum_{k=0}^\infty\frac{(\alpha+n)_k\,(\alpha+1)_k\,(\alpha+n)_k\,(1-\tfrac{n}{2})_k}
{(m+\alpha+n)_k\,(\alpha+\tfrac{3n}{2})_k\,(\alpha+\tfrac{n}{2}+1)_k\, k!}
\sum_{j=0}^\infty\frac{(\alpha+n+k)_j\,(1-\tfrac{n}{2}+k)_j\,(\tfrac{n}{2})_j}
{(\alpha+\tfrac{3n}{2}+k)_j\,(\alpha+\tfrac{n}{2}+1+k)_j\, j!}.
\end{align*}
Observe that the inner sum is a ${}_3F_2$ evaluated at $z=1$.
We transform this sum by Lemma \ref{LKummer}.
\begin{align*}
\sum_{j=0}^\infty
\frac{(\alpha+n+k)_j\,(1-\tfrac{n}{2}+k)_j\,(\tfrac{n}{2})_j}
{(\alpha+\tfrac{3n}{2}+k)_j\,(\alpha+\tfrac{n}{2}+1+k)_j\, j!}
&=\,
{}_3F_2\left[
\begin{aligned}
& \tfrac{n}{2},\,1-\tfrac{n}{2}+k,\,\alpha+n+k\\
& \alpha+\tfrac{n}{2}+1+k,\,\alpha+\tfrac{3n}{2}+k
\end{aligned}\,;
1\right]\\
&\hspace{-2.5cm}
=\frac{\Gamma(\alpha+\tfrac{3n}{2}+k)\Gamma(\alpha+n)}
{\Gamma(\alpha+n+k)\Gamma(\alpha+\tfrac{3n}{2})}\,
{}_3F_2\left[
\begin{aligned}
& \tfrac{n}{2},\,\alpha+n,\,1-\tfrac{n}{2}\\
& \alpha+\tfrac{n}{2}+1+k,\,\alpha+\tfrac{3n}{2}
\end{aligned}\,;
1\right]\\
&\hspace{-2.5cm}=
\frac{(\alpha+\tfrac{3n}{2})_k}{(\alpha+n)_k}\,
{}_3F_2\left[
\begin{aligned}
& \tfrac{n}{2},\,\alpha+n,\,1-\tfrac{n}{2}\\
& \alpha+\tfrac{n}{2}+1+k,\,\alpha+\tfrac{3n}{2}
\end{aligned}\,;
1\right],
\end{align*}
where in the last equality we use \eqref{Poch}.
We apply Lemma \ref{LKummer} for two reasons.
First to obtain $1-\tfrac{n}{2}$ as an upper parameter in ${}_3F_2$ so that
the sum will terminate when $n$ is even.
Second there will be some cancellations and formula for $I_m$ will simplify.
After the cancellations we obtain
\begin{align*}
I_m=&
\frac{\Gamma(\alpha+2n-1)\Gamma(\alpha+n)\Gamma(\tfrac{n}{2})}
{\Gamma^2(n-1)\Gamma(\alpha+\tfrac{3n}{2})}\,
\frac{\Gamma(m+n-1)}{\Gamma(m+\alpha+n)}\\
&\times
\sum_{k=0}^\infty\frac{(\alpha+1)_k\,(\alpha+n)_k\,(1-\tfrac{n}{2})_k}
{(m+\alpha+n)_k\,(\alpha+\tfrac{n}{2}+1)_k\, k!}\,
{}_3F_2\left[
\begin{aligned}
& \tfrac{n}{2},\,\alpha+n,\,1-\tfrac{n}{2}\\
& \alpha+\tfrac{n}{2}+1+k,\,\alpha+\tfrac{3n}{2}
\end{aligned}\,;1\right].
\end{align*}
This shows that $I_m$ can be written in the form \eqref{ImSeries}
and finishes the proof.
\end{proof}

We remark about the case $n$ is even.
In this case $A_k=0$ for $k\geq\tfrac{n}{2}$ because of
the term $(1-\tfrac{n}{2})_k$ and therefore the sum
in \eqref{ImSeries} is finite and is a rational function of $m$.
For $0\leq k<\tfrac{n}{2}$, the coefficients $A_k$ can also be computed
in finite steps since the ${}_3F_2$ in \eqref{Ak} terminates.

The exact values of the constants $A_k$ will not be important for us.
Nevertheless let us find a closed formula for the dominating term $A_0$.
When $k=0$, we can compute the ${}_3F_2$ in \eqref{Ak} with Lemma \ref{LDixon},
where we permute the upper parameters and take $a=\alpha+n$.
$${}_3F_2\left[
\begin{aligned}
& \alpha+n,\,\tfrac{n}{2},\,1-\tfrac{n}{2}\\
& \alpha+\tfrac{n}{2}+1,\,\alpha+\tfrac{3n}{2}
\end{aligned}\,;1\right]
=\frac{\Gamma(\tfrac{\alpha+n}{2}+1)\Gamma(\alpha+\tfrac{n}{2}+1)
\Gamma(\alpha+\tfrac{3n}{2})\Gamma(\tfrac{\alpha+n}{2})}
{\Gamma(\alpha+n+1)\Gamma(\tfrac{\alpha}{2}+1)\Gamma(\tfrac{\alpha}{2}+n)\Gamma(\alpha+n)}.
$$
After cancellations we obtain
\begin{equation}\label{A0}
A_0
=\frac{\Gamma(\tfrac{n}{2})\Gamma(\alpha+\tfrac{n}{2}+1)\Gamma(\alpha+2n-1)
\Gamma(\tfrac{\alpha+n}{2})\Gamma(\tfrac{\alpha+n}{2}+1)}
{\Gamma^2(n-1)\Gamma(\alpha+n+1)\Gamma(\tfrac{\alpha}{2}+1)\Gamma(\tfrac{\alpha}{2}+n)}.
\end{equation}

To estimate the reproducing kernel $\mathcal R_\alpha$, only finite part of the
series in Theorem \ref{TIm} will be sufficient.
In the next corollary we estimate the remainder part.

\begin{corollary}\label{CAsymA}
For $K=1,2,\dots$,
$$I_m
=\frac{\Gamma(m+n-1)}{\Gamma(m+\alpha+n)}
\Biggl(\sum_{k=0}^{K-1}\frac{A_k}{(m+\alpha+n)_k}+O\Bigl(\frac{1}{m^K}\Bigr)\Biggr)
\qquad (m\to\infty),$$
where $A_k$ is as in \eqref{Ak}.
The implied constant of the $O$ term depends only on $\alpha,n,K$.
\end{corollary}

\begin{proof}
The remainder term in the series in \eqref{ImSeries} is
\begin{align*}
R_K&:=\sum_{k=K}^\infty \frac{A_k}{(m+\alpha+n)_k}\\
&= C\sum_{k=K}^\infty\frac{(\alpha+1)_k\,(\alpha+n)_k\,(1-\tfrac{n}{2})_k}
{(m+\alpha+n)_k\,(\alpha+\tfrac{n}{2}+1)_k\, k!}\,
{}_3F_2\left[
\begin{aligned}
& \tfrac{n}{2},\,\alpha+n,\,1-\tfrac{n}{2}\\
& \alpha+\tfrac{n}{2}+1+k,\,\alpha+\tfrac{3n}{2}
\end{aligned}\,;1\right],
\end{align*}
where $C$ is a constant that depends only on $\alpha$ and $n$.
To estimate $R_K$, first note that by Lemma \ref{LF32Prop} (iii),
for every $k\geq 1$, we have
$$0\leq{}_3F_2\left[
\begin{aligned}
& \tfrac{n}{2},\,\alpha+n,\,1-\tfrac{n}{2}\\
& \alpha+\tfrac{n}{2}+1+k,\,\alpha+\tfrac{3n}{2}
\end{aligned}\,;1\right]\leq 1.$$
Next when $n$ is odd, by \eqref{Stirling}, there exists a constant $C$
depending only on $\alpha$ and $n$ such that for every $k\geq 1$,
$$\left|\frac{(\alpha+1)_k\,(1-\tfrac{n}{2})_k}{(\alpha+\tfrac{n}{2}+1)_k\, k!}\right|
\leq\frac{C}{k^n}.$$
The above estimate is certainly true when $n$ is even since left-hand side
vanishes for $k\geq\frac{n}{2}$.
Thus
\begin{align*}
|R_K|&\leq C \sum_{k=K}^\infty\frac{(\alpha+n)_k}{(m+\alpha+n)_k}\frac{1}{k^n}
=C\frac{(\alpha+n)_K}{(m+\alpha+n)_K}
\sum_{k=K}^\infty\frac{(\alpha+n+K)_{k-K}}{(m+\alpha+n+K)_{k-K}}\frac{1}{k^n}\\
&\leq C\frac{(\alpha+n)_K}{(m+\alpha+n)_K}\sum_{k=1}^\infty\frac{1}{k^n}
=C\frac{(\alpha+n)_K}{(m+\alpha+n)_K},
\end{align*}
since $(\alpha+n+K)_j\leq(m+\alpha+n+K)_j$ and $\sum_{k=1}^\infty 1/k^n$ converges.
Finally, because $(m+\alpha+n)_K\geq m^K$, we conclude that there exists a constant $C$
depending only on $\alpha,n$ and $K$ but independent of $m$ such that
$|R_K|\leq C/m^K$.
\end{proof}

Corollary \ref{CAsymA} shows that the series \eqref{ImSeries} is a convergent
asymptotic series
$$I_m\approx\frac{\Gamma(m+n-1)}{\Gamma(m+\alpha+n)}
\sum_{k=0}^\infty\frac{A_k}{(m+\alpha+n)_k}
\qquad (m\to\infty).$$
Our main interest is in $c_m(\alpha)=1/I_m$ and to invert $I_m$, we next write the
above asymptotic expansion in powers of $1/m$.
As a function of $m$, $(m+\alpha+n)_j$ is a polynomial of degree $j$ and
for $j\geq 1$, the rational function $1/(m+\alpha+n)_j$ has the Laurent series expansion
\begin{equation}\label{laurent}
\frac{1}{(m+\alpha+n)_j}=\sum_{k=j}^\infty \frac{C_k(j)}{m^k}
\qquad (m>\alpha+n+j-1),
\end{equation}
where $C_k(j)$ depends only on $\alpha$ and $n$.
When $j=0$, $(m+\alpha+n)_0=1$ and we set $C_0(0)=1$ and $C_k(0)=0$ for $k\geq 1$.
An explicit formula for the coefficients $C_k(j)$ can be written
(the partial fraction expansion of $1/(m+\alpha+n)_j$ has a simple form)
but we will not do this since exact values of these coefficients are not
important for our purposes.
For $k=0,1,2,\dots,$ define
\begin{equation}\label{Bk}
B_k=\sum_{j=0}^k A_jC_k(j).
\end{equation}
In particular $B_0=A_0$ with $A_0$ as in \eqref{A0}.

\begin{corollary}\label{CAsymB}
There exist constants $B_k$ $(k=0,1,2,\dots)$ depending only on $\alpha$ and $n$
such that for every $K=1,2,\dots$,
$$I_m
=\frac{\Gamma(m+n-1)}{\Gamma(m+\alpha+n)}
\Biggl(\sum_{k=0}^{K-1}\frac{B_k}{m^k}+O\Bigl(\frac{1}{m^K}\Bigr)\Biggr)
\qquad (m\to\infty).$$
The constants $B_k$ are given in \eqref{Bk}.
The implied constant of the $O$ term depends only on $\alpha,n,K$.
\end{corollary}

\begin{proof}
Fix $K$.
For $j=0,1,2,\dots,K-1$, by the Laurent series in \eqref{laurent}, we have
$$\frac{1}{(m+\alpha+n)_j}=
\sum_{k=j}^{K-1}\frac{C_k(j)}{m^k}+O\Bigl(\frac{1}{m^K}\Bigr)\qquad (m\to\infty).$$
Using Corollary \ref{CAsymA} and changing the orders of the sums we obtain
\begin{align*}
I_m
&=\frac{\Gamma(m+n-1)}{\Gamma(m+\alpha+n)}
\Biggl(\sum_{j=0}^{K-1}A_j\sum_{k=j}^{K-1}\frac{C_k(j)}{m^k}
+O\Bigl(\frac{1}{m^K}\Bigr)\Biggr)\\
&=\frac{\Gamma(m+n-1)}{\Gamma(m+\alpha+n)}
\Biggl(\sum_{k=0}^{K-1}\frac{1}{m^k}\sum_{j=0}^{k}A_jC_k(j)
+O\Bigl(\frac{1}{m^K}\Bigr)\Biggr)
\qquad (m\to\infty),
\end{align*}
which is the desired result by \eqref{Bk}.
\end{proof}

Once we have the asymptotic expansion of $I_m$ in powers of $1/m$, it
is straightforward to write the asymptotic expansion of $c_m(\alpha)=1/I_m$
(see, for example, \cite[p.~20]{O}).
Let $D_0=1/B_0$ (note that $B_0=A_0\neq 0$) and for $k\geq 1$, define $D_k$
recursively by the equation
\begin{equation}\label{Dk}
D_kB_0+D_{k-1}B_1+\dots+D_0B_k=0.
\end{equation}

\begin{corollary}\label{CAsymD}
There exist constants $D_k$ $(k=0,1,2,\dots)$ with $D_0>0$, depending only
on $\alpha$ and $n$ such that for every $K=1,2,\dots$,
$$c_m(\alpha)
=\frac{\Gamma(m+\alpha+n)}{\Gamma(m+n-1)}
\Biggl(\sum_{k=0}^{K-1}\frac{D_k}{m^k}+O\Bigl(\frac{1}{m^K}\Bigr)\Biggr)
\qquad (m\to\infty).$$
The constants $D_k$ ($k\geq 1$) are determined by \eqref{Dk}.
The implied constant of the $O$ term depends only on $\alpha,n,K$.
\end{corollary}

If we take $K=1$ in the above Corollary we obtain \eqref{Kis1}, since
$D_0=1/A_0$ with $A_0$ given in \eqref{A0}.

\section{Upper Estimates of Reproducing Kernels}\label{SE}

For a multi-index $\lambda=(\lambda_1,\dots,\lambda_n)$, where $\lambda_i$
are non-negative integers, and for a smooth function $f$ we write
$$\partial^\lambda f(x)=
\frac{\partial^{|\lambda|}}{\partial x_1^{\lambda_1}\dots\partial x_n^{\lambda_n}},$$
where $|\lambda|=\lambda_1+\dots+\lambda_n$.
The following estimate of the Euclidean Poisson kernel
$P(x,\eta)=\sum_{m=0}^\infty Z_m(x,\eta)$ is well-known (see, for example, \cite{K}).

\begin{lemma}\label{LPoisson}
Let $\lambda=(\lambda_1,\dots,\lambda_n)$ be a multi-index.
There exists a constant $C>0$ depending only on $\lambda$ and $n$ such that
for all $x\in\mathbb B$ and $\eta\in\mathbb S$,
$$|\partial^\lambda P(x,\eta)|\leq \frac{C}{|x-\eta|^{n-1+|\lambda|}}.$$
\end{lemma}

Using Lemma \ref{LPoisson} it is easy to estimate
$\sum_{m=1}^\infty mZ_m(x,\eta)=\tfrac{d}{dr}P(rx,\eta)\big\rvert_{r=1}$ and
applying the same argument repeatedly shows (for $k=1,2,\dots$)
\begin{equation}\label{mkEst}
\Bigl|\sum_{m=1}^\infty m^k Z_m(x,\eta)\Bigr|
\lesssim \frac{1}{|x-\eta|^{n-1+k}}.
\end{equation}
We need a more general version of the above estimate.

\begin{proposition}\label{PEstW}
Let $a_i,b_i\geq 0\ (i=1,2,\dots,k)$ and for $x\in\mathbb B,\eta\in\mathbb S$, let
$$W(x,\eta)=\sum_{m=1}^\infty\frac{\Gamma(m+a_1)}{\Gamma(m+b_1)}
\cdots\frac{\Gamma(m+a_k)}{\Gamma(m+b_k)}\,Z_m(x,\eta).$$
For a multi-index $\lambda$, set
$c=n-1+(a_1-b_1)+\cdots+(a_k-b_k)+|\lambda|$.
There exists a constant $C>0$ depending only on $a_i,b_i,n$ and $\lambda$
such that for all $x\in\mathbb B$ and $\eta\in\mathbb S$,
$$\bigl|\partial^\lambda W(x,\eta)\bigr|
\leq C\begin{cases}
\dfrac{1}{|x-\eta|^{c}},&\text{if $c>0$};\\
1+\log\dfrac{1}{|x-\eta|},&\text{if $c=0$};\\
1,&\text{if $c<0$}.
\end{cases}$$
\end{proposition}

Observe that by \eqref{Stirling}, the coefficient
$\prod_{i=1}^k\Gamma(m+a_i)/\Gamma(m+b_i)\sim m^{\sum_{i=1}^k(a_i-b_i)}$
as $m\to\infty$ and when $|\lambda|=0$, $\,c$ is $(n-1)$ plus the exponent of $m$. So
the above theorem generalizes \eqref{mkEst} by replacing $m^k$ with more
general coefficients that are asymptotic to $m^k$.
If there is differentiation, $c$ increases by the order of the derivative.
In the Euclidean case reproducing kernels $R_\alpha$ of Bergman spaces
(and also the generalized family of Bergman-Besov spaces with $\alpha\in\mathbb R$
in \cite{GKU}) are of the above form.

This proposition, in the form written above, is essentially proved in \cite[Lemma 7.4]{GKU}
following the methods of \cite[Section 3]{CR}, \cite[Chapter 7]{DS},
\cite[Section 2]{JP1}, \cite[Section 3]{LSR}) and \cite[Section 2]{Per}.
However, there is one detail that needs to be mentioned.
In \cite{GKU}, the summation starts from $m=0$ and therefore
$a_i,b_i$ are taken to be strictly positive.
Since the case $a_i$ or $b_i$ equal to 0 will be important for us and will be
repeatedly used later we give a proof of Proposition \ref{PEstW} below omitting
the steps that are already written in detail in \cite{GKU}.
We begin with the following special case.

\begin{lemma}\label{LEstb}
For $b\geq 0$ and $l=0,1,2,\dots$, let $$W(x,\eta)=\sum_{m=1}^\infty\frac{\Gamma(m+b+l)}{\Gamma(m+b)}
\,Z_m(x,\eta)\qquad (x\in\mathbb B, \eta\in\mathbb S).$$
For a multi-index $\lambda$, there exists a constant $C>0$ such that for all $x\in\mathbb B$, $\zeta\in\mathbb S$, $$|(\partial^\lambda W)(x,\eta)|\leq \frac{C}{|x-\eta|^{n-1+l+|\lambda|}}.$$
\end{lemma}

\begin{proof}
If $l=0$, then $W(x,\eta)=P(x,\eta)-1$ since $Z_0(x,\eta)=1$, and the result follows from Lemma \ref{LPoisson}. If $l\geq 1$, consider first $\lambda=(0,\dots,0)$. By homogeneity of $Z_m$, $$r^{b+l-1}P(rx,\eta)=\sum_{m=0}^\infty r^{m+b+l-1}Z_m(x,\eta) =r^{b+l-1}+\sum_{m=1}^\infty r^{m+b+l-1}Z_m(x,\eta).$$ Differentiating $l$ times with respect to $r$ and putting $r=1$, we obtain
\begin{equation}\label{derPoi}
\frac{\partial^l}{\partial r^l}\left(r^{b+l-1}P(rx,\eta)\right)\big|_{r=1}
=(b)_l+W(x,\eta).
\end{equation}
Applying Leibniz and chain rules to the left and using Lemma \ref{LPoisson} we obtain the desired result since highest order of differentiation applied to $P$ is $l$. For a general $\lambda$, we apply $\partial^\lambda$ to \eqref{derPoi}, change the orders of the derivatives and apply Lemma \ref{LPoisson}. See the proof of Lemma 7.2 in \cite{GKU} for more details.
\end{proof}

\begin{lemma}\label{LEstab}
For $a,b\geq 0$, let $$W(x,\eta)=\sum_{m=1}^\infty\frac{\Gamma(m+a)}{\Gamma(m+b)}
\,Z_m(x,\eta)\qquad (x\in\mathbb B, \eta\in\mathbb S).$$
For a multi-index $\lambda$, set $c=n-1+a-b+|\lambda|$. There exists a constant $C>0$ such that for all $x\in\mathbb B$, $\zeta\in\mathbb S$,
\begin{equation}\label{EstW}
|(\partial^\lambda W)(x,\eta)|
\leq C\begin{cases}
\dfrac{1}{|x-\eta|^c},&\text{if $c>0$};\\
1+\log\dfrac{1}{|x-\eta|},&\text{if $c=0$};\\ 1,&\text{if $c<0$}.
\end{cases}
\end{equation}
\end{lemma}

\begin{proof}
We first consider the case $a>0$. Pick a non-negative integer $l$ such that $b-a+l>0$ and let
\begin{equation}\label{Defg}
g(x,\eta)=\sum_{m=1}^\infty\frac{\Gamma(m+b+l)}{\Gamma(m+b)}
\,Z_m(x,\eta)\qquad (x\in\mathbb B, \eta\in\mathbb S).
\end{equation}
We can write $W$ as an integral in the form
\begin{equation}\label{DefW}
W(x,\eta)=\frac{1}{\Gamma(b-a+l)}
\int_0^1 g(rx,\eta)r^{a-1}(1-r)^{b-a+l-1}\,dr.
\end{equation}
This can be verified by integrating the series expansion of $g(rx,\eta)$ which uniformly converges on $0\leq r\leq 1$ for fixed $x\in\mathbb B$. Estimating $g$ with Lemma \ref{LEstb} we obtain $$|W(x,\eta)|\lesssim
\int_0^1\frac{r^{a-1}(1-r)^{b-a+l-1}}{|rx-\eta|^{n-1+l}}\,dr.$$
When $\lambda=(0,\dots,0)$ the result follows from Lemma \ref{LInt01}. In the case of a general multi-index $\lambda$, apply $\partial^\lambda$ to \eqref{DefW}, push $\partial^\lambda$ into the integral and apply chain rule. The result follows from applying first Lemma \ref{LEstb} and then Lemma \ref{LInt01}.

We now consider the case $a=0$ which we need to deal with seperately because of the convergence problem in the integral in \eqref{DefW}. If $b=0$, then $W(x,\eta)=P(x,\eta)-1$ and the lemma follows from Lemma \ref{LPoisson}. If $b>0$, let (we take $l=0$ in \eqref{Defg}) $$g(x,\eta)=\sum_{m=1}^\infty Z_m(x,\eta)=P(x,\eta)-1.$$ Then $g(rx,\eta)=O(r)$ as $r\to 0^+$, since for $r\leq r_0<1$, by \eqref{EstZm} and $|x|<1$,
\begin{equation}\label{gisO}
|g(rx,\eta)=r\Big|\sum_{m=1}^\infty r^{m-1}Z_m(x,\eta)\Big|
\leq Cr\sum_{m=1}^\infty r_0^{m-1} m^{n-2}=Cr
\end{equation}
As above $$W(x,\eta)=\frac{1}{\Gamma(b)}\int_0^1 g(rx,\eta)r^{-1}(1-r)^{b-1}\,dr,$$ where the integral absolutely converges. Writing $\int_0^1=\int_0^{r_0}+\int_{r_0}^1$, the first integral is $O(1)$, and the second integral is estimated by Lemmas \ref{LPoisson} and \ref{LInt01} as above. If $|\lambda|>0$, by the chain rule, $\partial^\lambda\big(g(rx,\eta)\big)=O(r)$ as $r\to 0^+$ and we argue similarly.
\end{proof}

\begin{proof}[Proof of Proposition \ref{PEstW}]
We use induction on $k$. If $k=1$, the theorem follows from Lemma \ref{LEstab}. We deduce the case $k=2$, the general case is shown similarly. Let $$W(x,\eta)=\sum_{m=1}^\infty\frac{\Gamma(m+a_1)}{\Gamma(m+b_1)}
\frac{\Gamma(m+a_2)}{\Gamma(m+b_2)}\,Z_m(x,\eta),$$
and $$g(x,\eta)=\sum_{m=1}^\infty\frac{\Gamma(m+a_1)}{\Gamma(m+b_1)}
\,Z_m(x,\eta).$$
Note that starting with the estimate $\partial^\lambda P$, we estimated $\partial^\lambda g$ in two steps with Lemmas \ref{LEstb} and \ref{LEstab}. To estimate $\partial^\lambda W$ we repeat the same arguments, the only difference is we replace the role of $P$ by $g$. As the first step choose a nonnegative integer $l$ such that
\begin{equation}\label{Defd}
d:=n-1+(a_1-b_1)+l>0 \quad\text{and}\quad b_2-a_2+l>0,
\end{equation}
and let $$h(x,\eta)=\sum_{m=1}^\infty\frac{\Gamma(m+a_1)}{\Gamma(m+b_1)}
\frac{\Gamma(m+b_2+l)}{\Gamma(m+b_2)}\,Z_m(x,\eta).$$
We have (corresponding to \eqref{derPoi}) $$h(x,\eta)=\frac{\partial^l}{\partial r^l}\big(r^{b_2+l-1}g(rx,\eta)\big)\big|_{r=1}$$ and retracing the proof of Lemma \ref{LEstb} with replacing references to Lemma \ref{LPoisson} by Lemma \ref{LEstab} we see that
\begin{equation}\label{Estbetah}
|(\partial^\lambda h)(x,\eta)|\lesssim \frac{1}{|x-\eta|^{d+|\lambda|}}.
\end{equation}
Note that highest order of derivative applied to $g$ is $d+|\lambda|$ which is positive by the first assumption in \eqref{Defd}. For the second step observe that (corresponding to \eqref{DefW}) $$W(x,\eta)=\frac{1}{\Gamma(b_2-a_2+l)}
\int_0^1 h(rx,\eta)r^{a_2-1}(1-r)^{b_2-a_2+l-1}\,dr.$$
Retracing the proof of Lemma \ref{LEstab} with replacing references to Lemma \ref{LEstb} by the estimate \eqref{Estbetah} proves the case $k=2$. In case $a_2=0$ we also use the fact that $h(rx,\eta)=O(r)$ as $r\to 0^+$ which can be shown as in \eqref{gisO} with using also \eqref{Stirling}.
\end{proof}

In the next corollary we allow the second variable of $W$ to be
in $\overline{\mathbb B}$.

\begin{corollary}\label{CEst}
Let $a_i,b_i\geq 0\ (i=1,2,\dots,k)$ and for
$x\in\mathbb B$, $y\in\overline{\mathbb B}$, let
$$W(x,y)=\sum_{m=1}^\infty\frac{\Gamma(m+a_1)}{\Gamma(m+b_1)}
\cdots\frac{\Gamma(m+a_k)}{\Gamma(m+b_k)}\,Z_m(x,y).$$
For a multi-index $\lambda$, set
$c=n-1+(a_1-b_1)+\cdots+(a_k-b_k)+|\lambda|$.
Tthere exists a constant $C>0$
such that for all $x\in\mathbb B$, $y\in\overline{\mathbb B}$,
$$\bigl|\partial^\lambda W(x,y)\bigr|
\leq C\begin{cases}
\dfrac{1}{[x,y]^{c}},&\text{if $c>0$};\\
1+\log\dfrac{1}{[x,y]},&\text{if $c=0$};\\
1,&\text{if $c<0$},
\end{cases}$$
where differentiation is applied to the first variable.
\end{corollary}

\begin{proof}
If $y=|y|\eta$ with $\eta\in\mathbb S$, then by homogeneity of $Z_m$ in
both variables, we have $Z_m(x,y)=Z_m(|y|x,\eta)$.
Therefore $W(x,y)=W(|y|x,\eta)$ and by the chain rule
$\partial^\lambda W(x,y)=|y|^\lambda(\partial^\lambda W)(|y|x,\eta)$.
Since $|y|\leq 1$, $\partial^\lambda W(x,y)$ is bounded from above by the
same terms given in Proposition \ref{PEstW} only with $|x-\eta|$ replaced
by $\big||y|x-\eta\big|$.
Because $\big||y|x-\eta\big|=[x,y]$ by \eqref{bracketxy}, the corollary follows.
\end{proof}

We now insert the factors $S_m(|x|)$ and $S_m(|y|)$
into the series.

\begin{theorem}\label{TEstSm}
Let $a_i,b_i\geq 0\ (i=1,2,\dots,k)$ and for
$x\in\mathbb B$, $y\in\overline{\mathbb B}$, let
$$h(x,y)=\sum_{m=1}^\infty\frac{\Gamma(m+a_1)}{\Gamma(m+b_1)}
\cdots\frac{\Gamma(m+a_k)}{\Gamma(m+b_k)}\,
S_m(|x|)S_m(|y|)\,Z_m(x,y).$$
Set $c=n-1+(a_1-b_1)+\cdots+(a_k-b_k)$.
There exists a constant $C>0$ depending only on $a_i,b_i$ and $n$
such that for all $x\in\mathbb B$, $y\in\overline{\mathbb B}$,
$$|h(x,y)|
\leq C \begin{cases}
\dfrac{1}{[x,y]^{c}},&\text{if $c>0$};\\
1+\log\dfrac{1}{[x,y]},&\text{if $c=0$};\\
1,&\text{if $c<0$}.
\end{cases}$$
\end{theorem}

\begin{proof}
For shortness, write $d_m=\prod_{i=1}^k \Gamma(m+a_i)/\Gamma(m+b_i)$.
Inserting the integral formula \eqref{IntSmr} for both $S_m(|x|)$ and $S_m(|y|)$
and changing the orders of the integrals and the series we obtain
\begin{align*}
h(x,y)=\int_0^1\int_0^1&
(1-t)^{\tfrac{n}{2}-1}(1-|x|^2t)^{\tfrac{n}{2}-1}
(1-\tau)^{\tfrac{n}{2}-1}(1-|y|^2\tau)^{\tfrac{n}{2}-1}\\
&\times\sum_{m=1}^\infty \frac{d_m}{B^2(m,n-1)}\, t^{m-1}\tau^{m-1} Z_m(x,y)\,d\tau dt.
\end{align*}
The interchange of the orders is possible since for fixed $x$ and $y$
the series uniformly converges for $0\leq t,\tau\leq 1$ by \eqref{Stirling},
homogeneity of $Z_m$ and \eqref{EstZm}.
When $0\leq t\leq 1/2$ or $0\leq\tau\leq 1/2$ above integrand is uniformly bounded
for every $x,y$ by the same reasoning.
When $1/2\leq t,\tau\leq 1$, observe that
$t^{m-1}\tau^{m-1}Z_m(x,y)=Z_m(tx,\tau y)/t\tau$.
Therefore, if we let
$$W(x,y)=\sum_{m=1}^\infty d_m\,\frac{\Gamma^2(m+n-1)}{\Gamma^2(m)}\,Z_m(x,y),$$
it holds that
$$|h(x,y)|\leq C+C\int_{\frac{1}{2}}^1\int_{\frac{1}{2}}^1
\bigl((1-t)(1-|x|^2t)(1-\tau)(1-|y|^2\tau)\bigr)^{\frac{n}{2}-1}
\frac{1}{t\tau}|W(tx,\tau y)|\,d\tau dt.$$
We can get rid of the factor $1/(t\tau)$ and obtain
\begin{equation}\label{Wtxtauy}
|h(x,y)|\lesssim 1+\int_0^1\int_0^1
\bigl((1-t)(1-|x|^2t)(1-\tau)(1-|y|^2\tau)\bigr)^{\frac{n}{2}-1}
|W(tx,\tau y)|\,d\tau dt.
\end{equation}
We estimate $W(x,y)$ with Corollary \ref{CEst} in three cases depending
on the sign of $c+2(n-1)$, since the coefficient in the definition
of $W$ is $d_m\Gamma^2(m+n-1)/\Gamma^2(m)$.

Case (i). If $c+2(n-1)>0$, then by Corollary \ref{CEst},
$$|h(x,y)|\lesssim 1+\int_0^1\int_0^1
\frac{\bigl((1-t)(1-|x|^2t)(1-\tau)(1-|y|^2\tau)\bigr)^{\frac{n}{2}-1}}
{[tx,\tau y]^{c+2(n-1)}}\,d\tau dt.$$
Note that for $0\leq t\leq 1$,
\begin{equation}\label{Ineq1}
1-t\leq 1-|x|^2t\leq 1-|x|^2t^2\leq 2(1-|x|t)\leq 2(1-|x||y|t\tau)\leq 2[tx,\tau y],
\end{equation}
where the last inequality follows from \eqref{bracketEst}.
Similarly,
\begin{equation}\label{Ineq2}
1-\tau\leq 1-|y|^2\tau\leq 2[tx,\tau y].
\end{equation}
Hence
\begin{equation}\label{Esthxy}
|h(x,y)|\lesssim 1+\int_0^1\int_0^1
\frac{1}{[tx,\tau y]^{c+2}}\,d\tau dt
=1+\int_0^1\int_0^1
\frac{1}{|\tau t|y|x-\eta|^{c+2}}\,d\tau dt,
\end{equation}
where in the last equality we write $y=|y|\eta$, $\eta\in\mathbb S$
and use \eqref{bracketxy}.
We next estimate the inner integral with Lemma \ref{LInt01}.
This requires to consider three subcases.

If $c>-1$, applying Lemma \ref{LInt01} to the inner integral in \eqref{Esthxy} we obtain
$$|h(x,y)|\lesssim 1+\int_0^1
\frac{dt}{\bigl|t|y|x-\eta\bigr|^{c+1}}.$$
Applying Lemma \ref{LInt01} one more time in the three cases $c>0$, $c=0$ and $-1<c<0$
(and noting that $\bigl||y|x-\eta\bigr|=[x,y]$ by \eqref{bracketxy}) proves the theorem when $c>-1$.

If $c=-1$, applying Lemma \ref{LInt01} to the inner integral in \eqref{Esthxy} shows
$$|h(x,y)|\lesssim 1+\int_0^1 \Bigl(1+\log\frac{1}{\bigl|t|y|x-\eta\bigr|}\Bigr)\,dt
\leq 1+\int_0^1 \Bigl(1+\log\frac{1}{1-t}\Bigr)\,dt\lesssim 1,$$
since $\bigl|t|y|x-\eta\bigr|\geq 1-t|y||x|\geq 1-t$, and last integral is finite.

If $-2(n-1)<c<-1$, applying Lemma \ref{LInt01} to the inner integral in \eqref{Esthxy},
we obtain
$$|h(x,y)|\lesssim 1+\int_0^1\,dt\lesssim 1.$$
Thus the theorem holds when $c>-2(n-1)$.

Case (ii). If $c+2(n-1)=0$, estimating $W(tx,\tau y)$ in \eqref{Wtxtauy}
with Corollary \ref{CEst}, we obtain
$$|h(x,y)|\lesssim 1+\int_0^1\int_0^1
\bigl((1-t)(1-|x|^2t)(1-\tau)(1-|y|^2\tau)\bigr)^{\frac{n}{2}-1}
\Bigl(1+\log\frac{1}{[tx,\tau y]}\Bigr)\,d\tau dt.$$
Since $[tx,\tau y]\geq 1-t\tau|x||y|\geq 1-t$ by \eqref{bracketEst}, and
the other terms in the integrand are bounded from above by 1, we have
$$|h(x,y)|\lesssim 1+\int_0^1\log\frac{1}{1-t}\,dt\lesssim 1.$$

Case (iii). If $c+2(n-1)<0$, then by \eqref{Wtxtauy} and Corollary \ref{CEst},
\begin{equation*}
|h(x,y)|\lesssim  1+\int_0^1\int_0^1
\bigl((1-t)(1-|x|^2t)(1-\tau)(1-|y|^2\tau)\bigr)^{\frac{n}{2}-1}\,d\tau dt\lesssim 1,
\end{equation*}
since again the integrand is bounded by $1$.
This finishes the proof.
\end{proof}

Using Theorem \ref{TEstSm} we can find upper bounds for the reproducing kernels of
both Bergman and Hardy spaces of $\mathcal{H}$-harmonic functions on $\mathbb B$.
We begin with Hardy spaces.
The reproducing kernel $\mathcal{K}(x,y)$ of the Hardy space $\mathcal{H}^2$
has the series expansion \cite[Theorem 2.5]{St3}
$$\mathcal{K}(x,y)=\sum_{m=0}^\infty S_m(|x|)S_m(|y|)\,Z_m(x,y) \qquad (x,y\in\mathbb B).$$
The following corollary of Theorem \ref{TEstSm} improves the estimate
of $\mathcal{K}(x,y)$ given in \cite[Theorem 3.2]{St3}.
We note that $\mathcal{K}(x,y)\geq 0$ by \cite[Theorem 2.1]{St3}).
\begin{corollary}
There exists a constant $C>0$ depending only on $n$ such that
for all $x,y\in\mathbb B$,
$$\mathcal{K}(x,y)\leq\frac{C}{[x,y]^{n-1}}.$$
Moreover the exponent $(n-1)$ is non-improvable.
\end{corollary}

\begin{proof}
By Theorem \ref{TEstSm} and the fact that $[x,y]\leq 2$,
$$|\mathcal{K}(x,y)|=\Bigl|1+\sum_{m=1}^\infty S_m(|x|)S_m(|y|)\,Z_m(x,y)\Bigr|
\lesssim 1+\frac{1}{[x,y]^{n-1}}
\lesssim \frac{1}{[x,y]^{n-1}}.$$
On the other hand, when $y=x$,
$$\mathcal K(x,x)\gtrsim 1+\sum_{m=1}^\infty m^{n-2}|x|^{2m}
\sim \frac{1}{(1-|x|^2)^{n-1}}=\frac{1}{[x,x]^{n-1}},$$
since $S_m(r)\geq 1$ by Lemma \ref{LSmrEst} and
$Z_m(x,x)\sim |x|^{2m}m^{n-2}$ by \eqref{EstZm}.
\end{proof}

We next deal with Bergman spaces and prove Theorem \ref{TEstBergman}.

\begin{proof}[Proof of Theorem \ref{TEstBergman}]
Pick a positive integer $K$ such that $K>\alpha+2n-2$.
By \eqref{RepKer} and Corollary \ref{CAsymD},
$$\mathcal R_\alpha(x,y)
=1+\sum_{m=1}^\infty\frac{\Gamma(m+\alpha+n)}{\Gamma(m+n-1)}
\biggl(\sum_{k=0}^{K-1}\frac{D_k}{m^k}+\delta_m(K)\biggr)
S_m(|x|)S_m(|y|)\,Z_m(x,y),$$
where $\delta_m(K)=O(1/m^K)$, $m\to\infty$.

We first show that the last part with coefficient $\delta_m(K)$ is uniformly
bounded for all $x,y\in\mathbb B$.
Note that by Lemma \ref{LSmrEst},
$S_m(r)\lesssim m^{\tfrac{n}{2}-1}$, and
by homogeneity and \eqref{EstZm}, $|Z_m(x,y)|\lesssim|x|^m|y|^m m^{n-2}\lesssim m^{n-2}$.
Therefore using also \eqref{Stirling},
$$\biggl|\sum_{m=1}^\infty
\frac{\Gamma(m+\alpha+n)}{\Gamma(m+n-1)}
\frac{1}{m^K} S_m(|x|)S_m(|y|)\,Z_m(x,y)\biggr|
\lesssim \sum_{m=1}^\infty\frac{1}{m^{K-(\alpha+2n-3)}}
\lesssim 1,$$
by the choice of $K$.

Next, for $k=0,1,\dots,K-1$, writing $1/m^k=\Gamma^k(m)/\Gamma^k(m+1)$ and
applying Theorem \ref{TEstSm} with $c=n-1+(\alpha+1)-k=\alpha+n-k$, we see that
$$\biggl|\sum_{m=1}^\infty\frac{\Gamma(m+\alpha+n)}{\Gamma(m+n-1)}
\frac{1}{m^k}S_m(|x|)S_m(|y|)\,Z_m(x,y)\biggr|
\lesssim\begin{cases}
\dfrac{1}{[x,y]^{\alpha+n-k}},&\text{if $0\leq k<\alpha+n$};\\
1+\log\dfrac{1}{[x,y]},&\text{if $k=\alpha+n$};\\
1,&\text{if $k>\alpha+n$}.
\end{cases}$$
Since $0<[x,y]\leq 2$ the term with highest growth rate is
$1/[x,y]^{\alpha+n}$ occurring when $k=0$.
We conclude that
$|\mathcal R_\alpha(x,y)|\lesssim 1/[x,y]^{\alpha+n}$.
The exponent $(\alpha+n)$ is non-improvable by Lemma \ref{LRlower} below.
\end{proof}

\section{First-Order Partial Derivatives of Reproducing Kernels}\label{SPar}

In this section we estimate first-order partial derivatives of $\mathcal R_\alpha$.
We begin with the following lemma.

\begin{lemma}\label{LDerZm}
There exists a constant $C>0$ depending only on $n$ such that for all $x,y\in\mathbb B$
and $i=1,2,\dots,n$,
$$\Bigl|\frac{\partial}{\partial x_i}Z_m(x,y)\Bigr|\leq Cm^n|x|^{m-1}|y|^m.$$
\end{lemma}

It appears that this lemma can be improved by replacing $m^n$ with $m^{n-1}$,
but for our purposes the above estimate will be sufficient.

\begin{proof}
Write $x=|x|\zeta$, $y=|y|\eta$ with $\zeta,\eta\in\mathbb S$.
Comparing the expansions of $Z_m$ in \cite[Theorem 5.38]{ABR} and
the Gegenbauer (ultraspherical) polynomials in \cite[18.5.10]{OLBC} shows
that (when $n\geq 3$)
$$Z_m(x,y)=\frac{n+2m-2}{n-2}|x|^m|y|^m\,
C_m^{(n/2-1)}\Bigl(\bigl\langle\frac{x}{|x|},\frac{y}{|y|}\bigr\rangle\Bigr).$$
Differentiating and using
$\frac{d}{dt}C_m^{(n/2-1)}(t)=(n-2)C_{m-1}^{(n/2)}(t)$
(see \cite[18.9.19]{OLBC}) shows
\begin{align*}
\frac{\partial}{\partial x_i}Z_m(x,y)=\frac{n+2m-2}{n-2}|x|^{m-1}|y|^m
&\Bigl(m\zeta_i\, C_m^{(n/2-1)}(\langle\zeta,\eta\rangle)\\
&+(n-2)(\eta_i-\zeta_i\langle\zeta,\eta\rangle)\,
C_{m-1}^{(n/2)}(\langle\zeta,\eta\rangle)\Bigr).
\end{align*}
By \cite[18.14.4]{OLBC}, for $-1\leq t\leq 1$, $|C_m^{(\lambda)}(t)|\leq(2\lambda)_m/m!$.
Using this and \eqref{Stirling} gives the desired result.
\end{proof}

We state the next lemma as a preparation to the proof of
Theorem \ref{TPartial}.

\begin{lemma}\label{LPrep}
Let $\alpha>-1$ and $c_m(\alpha)$ be as in \eqref{cm}.
Let
$$W(x,y)=\sum_{m=1}^\infty c_m(\alpha)
\frac{\Gamma^2(m+n-1)}{\Gamma^2(m)}\,Z_m(x,y).$$
There exists a constant $C>0$ depending only $\alpha$ and $n$ such that
for all $x,y\in\mathbb B$,
\begin{enumerate}
\item[(i)] $|W(x,y)|\leq\dfrac{C}{[x,y]^{\alpha+n+2(n-1)}}$,
\item[(ii)]$\Big|\dfrac{\partial}{\partial x_i}W(x,y)\Big|
\leq\dfrac{C}{[x,y]^{\alpha+n+2(n-1)+1}} \qquad(i=1,2,\dots,n).$
\end{enumerate}
\end{lemma}

\begin{proof}
Since by \eqref{gammamAsym}, the coefficient
$c_m(\alpha)\Gamma^2(m+n-1)/\Gamma^2(m)\sim m^{\alpha+1+2(n-1)}$ as $m\to\infty$,
the lemma follows from Corollary \ref{CEst}.
Nevertheless we provide the details.
Pick $K$ with $K>\alpha+3n-2$.
By Corollary \ref{CAsymD},
\begin{equation}\label{DefineW}
W(x,y)=\sum_{m=1}^\infty\frac{\Gamma(m+\alpha+n)}{\Gamma(m+n-1)}
\frac{\Gamma^2(m+n-1)}{\Gamma^2(m)}
\biggl(\sum_{k=0}^{K-1}\frac{D_k}{m^k}+\delta_m(K)\biggr)\,Z_m(x,y),
\end{equation}
where $\delta_m(K)=O(1/m^K)$.
Corollary \ref{CEst} implies that for the dominating term
corresponding to $k=0$,
\begin{equation}\label{sumless}
\biggl|\sum_{m=1}^\infty\frac{\Gamma(m+\alpha+n)}{\Gamma(m+n-1)}
\frac{\Gamma^2(m+n-1)}{\Gamma^2(m)}\frac{1}{m^k}\,Z_m(x,y)\biggr|
\lesssim \frac{1}{[x,y]^{\alpha+n+2(n-1)}}
\end{equation}
and for $k=1,2,\dots,K-1$, a better estimate holds but since $[x,y]\leq 2$,
\eqref{sumless} is true for these $k$'s also.
For the part involving coefficient $\delta_m(K)$, by \eqref{Stirling} and \eqref{EstZm},
$$\sum_{m=1}^\infty\biggl|\frac{\Gamma(m+\alpha+n)}{\Gamma(m+n-1)}
\frac{\Gamma^2(m+n-1)}{\Gamma^2(m)}\frac{1}{m^K}\,Z_m(x,y)\biggr|
\lesssim \sum_{m=1}^\infty \frac{1}{m^{K-(\alpha+3n-3)}}\lesssim 1.$$

To estimate the partial derivative, pick $K>\alpha+3n$ and differentiate \eqref{DefineW}.
By Corollary \ref{CEst}, for all $k=0,1,\dots,K-1$,
$$\biggl|\frac{\partial}{\partial x_i}\sum_{m=1}^\infty\frac{\Gamma(m+\alpha+n)}{\Gamma(m+n-1)}
\frac{\Gamma^2(m+n-1)}{\Gamma^2(m)}\frac{1}{m^k}\,Z_m(x,y)\biggr|
\lesssim \frac{1}{[x,y]^{\alpha+n+2(n-1)+1}}.$$
For the part involving coefficient $\delta_m(K)$, by Lemma \ref{LDerZm} and \eqref{Stirling},
\begin{equation*}
\sum_{m=1}^\infty\biggl|\frac{\Gamma(m+\alpha+n)}{\Gamma(m+n-1)}
\frac{\Gamma^2(m+n-1)}{\Gamma^2(m)}\frac{1}{m^K}\frac{\partial}{\partial x_i}Z_m(x,y)\biggr|
\lesssim \sum_{m=1}^\infty \frac{1}{m^{K-(\alpha+3n-1)}}\lesssim 1.\qedhere
\end{equation*}
\end{proof}

We are  now ready to prove Theorem \ref{TPartial}.

\begin{proof}[Proof of Theorem \ref{TPartial}]
Differentiating \eqref{RepKer}, we have
\begin{align*}
\frac{\partial}{\partial x_i}\mathcal R_\alpha(x,y)
=&\sum_{m=1}^\infty c_m(\alpha)\frac{\partial}{\partial x_i}S_m(|x|)
S_m(|y|)\,Z_m(x,y)\\
+&\sum_{m=1}^\infty c_m(\alpha)S_m(|x|)S_m(|y|)
\frac{\partial}{\partial x_i}Z_m(x,y)
=:T_1+T_2.
\end{align*}
The method we will use to estimate $T_1$ and $T_2$ is similar to the proof
of Theorem \ref{TEstSm} where every step is justified.
Here we will be brief and will not repeat the same arguments unless there is
an essential difference.

We first estimate $T_1$.
Using the integral formula \eqref{IntSmr} for both $S_m(|x|)$ and $S_m(|y|)$
and differentiating $S_m(|x|)$ under the integral sign shows
\begin{align*}
T_1=-(n-2)x_i\int_0^1\int_0^1&
(1-t)^{\tfrac{n}{2}-1}(1-|x|^2t)^{\tfrac{n}{2}-2}
(1-\tau)^{\tfrac{n}{2}-1}(1-|y|^2\tau)^{\tfrac{n}{2}-1}\\
&\times\sum_{m=1}^\infty \frac{c_m(\alpha)}{B^2(m,n-1)}\, t^m\tau^{m-1} Z_m(x,y)\,d\tau dt.
\end{align*}
Observe that the power of $(1-|x|^2t)$ is $n/2-2$ which is negative when $n=3$
and this case should be treated separately.
Define $W(x,y)$ as in Lemma \ref{LPrep}.
Using that $t^m\tau^{m-1}Z_m(x,y)=Z_m(tx,\tau y)/\tau$ we obtain
$$|T_1|\leq C+C\int_0^1\int_0^1
\bigl((1-t)(1-\tau)(1-|y|^2\tau)\bigr)^{\tfrac{n}{2}-1}
(1-|x|^2t)^{\tfrac{n}{2}-2}
|W(tx,\tau y)|\,d\tau dt,$$
where we can get rid of the factor $1/\tau$ in the same way as done in
the proof of Theorem \ref{TEstSm}.
Applying Lemma \ref{LPrep} (i) we deduce
$$|T_1|\lesssim 1+\int_0^1\int_0^1
\frac{\bigl((1-t)(1-\tau)(1-|y|^2\tau)\bigr)^{\tfrac{n}{2}-1}(1-|x|^2t)^{\tfrac{n}{2}-2}}
{[tx,\tau y]^{\alpha+n+2(n-1)}}\,d\tau dt.$$
When $n\geq 4$, using the inequalities in \eqref{Ineq1} and \eqref{Ineq2},
we conclude
\begin{equation}\label{T1}
|T_1|\lesssim 1+\int_0^1\int_0^1
\frac{1}{[tx,\tau y]^{\alpha+n+3}}\,d\tau dt
\lesssim \frac{1}{[x,y]^{\alpha+n+1}},
\end{equation}
where in the last inequality we use \eqref{bracketxy} and Lemma \ref{LInt01} twice.
When $n=3$, the power of $1-|x|^2t$ is negative and we repeat the same argument
with the only change that we use the inequality $1-|x|^2t\geq 1-t$.
We obtain \eqref{T1} again.

We now estimate $T_2$.
Applying the integral formula \eqref{IntSmr} twice shows
\begin{align*}
T_2=\int_0^1\int_0^1&
(1-t)^{\tfrac{n}{2}-1}(1-|x|^2t)^{\tfrac{n}{2}-1}
(1-\tau)^{\tfrac{n}{2}-1}(1-|y|^2\tau)^{\tfrac{n}{2}-1}\\
&\times\sum_{m=1}^\infty \frac{c_m(\alpha)}{B^2(m,n-1)}\, t^{m-1}\tau^{m-1}
\frac{\partial}{\partial x_i}Z_m(x,y)\,d\tau dt.
\end{align*}
We have
$$t^{m-1}\tau^{m-1}\,\frac{\partial}{\partial x_i} Z_m(x,y)
=\frac{1}{\tau}\Bigl(\frac{\partial}{\partial x_i}Z_m\Bigr)(tx,\tau y)$$
since $\tfrac{\partial }{\partial x_i}Z_m$ is homogeneous of degree
$m-1$ in the $x$-variable and degree $m$ in the $y$-variable.
Thus with $W$ as in Lemma \ref{LPrep},
$$|T_2|\leq C+C\int_0^1\int_0^1
\bigl((1-t)(1-|x|^2t)(1-\tau)(1-|y|^2\tau)\bigr)^{\tfrac{n}{2}-1}
\Bigl|\Bigl(\frac{\partial}{\partial x_i}W\Bigr)(tx,\tau y)\Bigr|\,d\tau dt,$$
where we can get rid of the factor $1/\tau$ as before.
Applying Lemma \ref{LPrep} (ii) shows
$$|T_2|\lesssim 1+\int_0^1\int_0^1
\frac{\bigl((1-t)(1-|x|^2t)(1-\tau)(1-|y|^2\tau)\bigr)^{\tfrac{n}{2}-1}}
{[tx,\tau y]^{\alpha+n+2(n-1)+1}}\,d\tau dt,$$
and first using \eqref{Ineq1} and \eqref{Ineq2} and then \eqref{bracketxy}
and Lemma \ref{LInt01} twice gives
\begin{equation*}
|T_2|\lesssim 1+\int_0^1\int_0^1
\frac{1}{[tx,\tau y]^{\alpha+n+3}}\,d\tau dt
\lesssim \frac{1}{[x,y]^{\alpha+n+1}}.
\end{equation*}
With the estimate of $T_1$ in \eqref{T1} and $T_2$ above, the proof is complete.
\end{proof}

\section{Weighted Integrals of Powers of Reproducing Kernels}\label{SInt}

In this section we prove the two-sided estimate given in Theorem \ref{TIntegral}.
The upper estimate part immediately follows from Theorem \ref{TEstBergman}
and the main problem is to obtain the lower estimate.
For this we first estimate $\mathcal R_\alpha(x,y)$ from below
when $y$ is a positive multiple of $x$.

\begin{lemma}\label{LRlower}
Let $\alpha>-1$.
There exists a constant $C>0$ depending only $\alpha$ and $n$ such that
for all $0\leq r,\rho<1$ and $\zeta\in\mathbb S$,
$$\mathcal R_\alpha(r\zeta,\rho\zeta)\geq \frac{C}{(1-r\rho)^{\alpha+n}}.$$
\end{lemma}

\begin{proof}
Note that every term in the sum
$$\mathcal R_\alpha(r\zeta,\rho\zeta)=
1+\sum_{m=1}^\infty c_m(\alpha)S_m(r)S_m(\rho)\,Z_m(r\zeta,\rho\zeta)$$
is non-negative because $S_m(r)\geq 1$ by Lemma \ref{LSmrEst},
$c_m(\alpha)>0\,$ by \eqref{cm}, and
$Z_m(r\zeta,\rho\zeta)=(r\rho)^m\text{dim}\mathcal H_m\geq0$ by \eqref{EstZm}.
Since $c_m(\alpha)\sim m^{\alpha+1}$ as $m\to\infty$ by \eqref{gammamAsym}
and $c_m(\alpha)$ is strictly positive,
there exists a constant $C$ such that $c_m(\alpha)\geq Cm^{\alpha+1}$ for all $m\geq 1$.
Therefore using also \eqref{EstZm}, we have
\begin{equation*}
\mathcal R_\alpha(r\zeta,\rho\zeta)
\geq C\Bigl(1+\sum_{m=1}^\infty m^{\alpha+n-1}(r\rho)^m\Bigr)
\geq\frac{C}{(1-r\rho)^{\alpha+n}}.
\qedhere
\end{equation*}
\end{proof}

Lemma \ref{LRlower} is not sufficient to estimate the integral in
Theorem \ref{TIntegral} from below.
Our aim is to extend Lemma \ref{LRlower} to a larger domain of $y$.

For an orthogonal transformation $U:\mathbb R^n\to\mathbb R^n$,
by \cite[Proposition 5.27 (c)]{ABR}, $Z_m(Ux,Uy)=Z_m(x,y)$ and therefore
$\mathcal R_\alpha(Ux,Uy)=\mathcal R_\alpha(x,y)$.
Since there exists such a $U$ that $Ux=(|x|,0,\dots,0)$, for
notational simplicity, we consider only $x$ of the form $(r,0,\dots,0)$.
Let $e_1=(1,0,\dots,0)$.
For $y\in\mathbb R^n$, we write $y=(y_1,\hat y)$ with
$\hat y=(y_2,\dots,y_n)\in\mathbb R^{n-1}$.
For $s>0$, define the nontangential approach region
$$\Omega_s=\{\,y=(y_1,\hat y)\in\mathbb B\,:\,0<y_1<1,\ |\hat y|<s(1-y_1)\,\}.$$
The set $\Omega_s$ is intersection of a cone with vertex at $e_1$,
$\mathbb B$, and right half-space.

\begin{proposition}\label{Plower}
Let $\alpha>-1$.
There exists an $s<1/2$ and a constant $C>0$ depending only
on $\alpha,n,s$ such that for every $x=re_1$, $0\leq r<1$ and
$y\in\Omega_s$,
$$\mathcal R_\alpha(x,y)\geq\frac{C}{(1-ry_1)^{\alpha+n}}.$$
\end{proposition}

\begin{proof}
We follow \cite[Proposition 5]{Mia} and \cite[Theorem 7.2]{GKU}.
Given $y=(y_1,\hat y)\in\Omega_s$, with $s<1/2$,
let $y_p=(y_1,0)$ be the projection onto $y_1$-axis.
By Lemma \ref{LRlower}, there exists a constant $C_1>0$ such that
$$\mathcal R_\alpha(x,y_p)\geq\frac{C_1}{(1-ry_1)^{\alpha+n}}.$$
By the mean-value theorem of calculus
$$\mathcal R_\alpha(x,y)\geq\mathcal R_\alpha(x,y_p)-
\max_{z\in L(y_p,y)}\bigl|\nabla_z\mathcal R_\alpha(x,z)\bigr|\,|y-y_p|,$$
where $L(y_p,y)$ is the line segment joining $y_p$ and $y$.
By Theorem \ref{TPartial}, there exists a constant $C_2>0$ such that
(we also use the fact that $\mathcal R_\alpha(\cdot,\cdot)$ is symmetric with
respect to its two variables)
$$\bigl|\nabla_z\mathcal R_\alpha(x,z)\bigr|\leq\frac{C_2}{[x,z]^{\alpha+n+1}}
\leq\frac{C_2}{(1-r|z|)^{\alpha+n+1}},$$
where in the last inequality we use \eqref{bracketEst}.
Since $s<1/2$ and $L(y_p,y)\subset\Omega_s$, for $z\in L(y_p,y)$
we have $|z|\leq y_1+s(1-y_1)\leq (1+y_1)/2$. Therefore
$$1-r|z|\geq 1-\frac{r}{2}-\frac{ry_1}{2}\geq \frac{1-ry_1}{2}.$$
Combining above we deduce
$$\mathcal R_\alpha(x,y)\geq\frac{C_1}{(1-ry_1)^{\alpha+n}}
-\frac{2^{\alpha+n+1}C_2|y-y_p|}{(1-ry_1)^{\alpha+n+1}}.$$
Because $|y-y_p|=|\hat y|<s(1-y_1)<s(1-ry_1)$,
we can find a sufficiently small $s$ such that the lemma holds.
\end{proof}

We are now ready to prove Theorem \ref{TIntegral}.
\begin{proof}[Proof of Theorem \ref{TIntegral}]
By Theorem \ref{TEstBergman},
$$J(x)
:=\int_\mathbb B\bigl|\mathcal R_\alpha(x,y)\bigr|^{p}(1-|y|^2)^\beta\,d\nu(y)
\lesssim \int_\mathbb B\frac{(1-|y|^2)^\beta}{[x,y]^{p(\alpha+n)}}\,d\nu(y)$$
and the upper estimate follows from Lemma \ref{LIntBracket}.

For the lower estimate first note that since $\mathcal R_\alpha(\cdot,\cdot)$,
$1-|y|^2$ and $d\nu(y)$ are all invariant
under orthogonal transformations, without loss of generality, we can
assume $x=(r,0,\dots,0)$, $0\leq r<1$.
Let $s$ be as given in Proposition \ref{Plower}.
Then
$$J(x)
\geq\int_{\Omega_s}\bigl|\mathcal R_\alpha(x,y)\bigr|^{p}(1-|y|^2)^\beta\,d\nu(y)
\gtrsim\int_{\Omega_s}\frac{(1-|y|^2)^\beta}{(1-ry_1)^{p(\alpha+n)}}\,d\nu(y).$$
For $y=(y_1,\hat y)\in\Omega_s$, we have $(1-|y|^2)\sim(1-y_1)$.
This is true since first, $1-y_1\geq 1-|y|\geq (1-|y|^2)/2$.
Second, $|y|\leq y_1+|\hat y|\leq y_1+s(1-y_1)$ and so
$1-|y|^2\geq 1-|y|\geq (1-s)(1-y_1)$.
Thus iterating the integral over $\Omega_s$ we obtain
$$J(x)
\gtrsim\int_0^1\int_{s(1-y_1)\mathbb B_{n-1}}
\frac{(1-y_1)^\beta}{(1-ry_1))^{p(\alpha+n)}}\,d\nu_{n-1}(\hat y)\,dy_1
\sim\int_0^1\frac{(1-y_1)^{\beta+n-1}}{(1-ry_1)^{p(\alpha+n)}}\,dy_1,$$
where $\mathbb B_{n-1}$ is the $(n-1)$-dimensional unit ball,
$\nu_{n-1}$ is the $(n-1)$-dimensional volume measure and
$\nu_{n-1}\bigl(s(1-y_1)\mathbb B_{n-1}\bigr)\sim(1-y_1)^{n-1}$.
The lower estimate now follows from Lemma \ref{LInt01r}.
\end{proof}

\section{Bergman Projection}\label{SProj}

In this section we prove Theorems \ref{Tprojection} and \ref{TProjBloch}.
With the estimates obtained earlier this will be straightforward and similar
to the proof of \cite[Theorem 1.9]{HKZ}.
We provide the details for completeness.

We first recall that for every $f\in\mathcal H(\mathbb B)$, there exists a unique
sequence of homogeneous harmonic polynomials $p_m\in H_m(\mathbb R^n)$ such that
for all $x\in\mathbb B$,
\begin{equation}\label{seriesH}
f(x)=\sum_{m=0}^\infty S_m(|x|)\,p_m(x),
\end{equation}
where the above series converges absolutely and
uniformly on compact subsets of $\mathbb B$.
A proof of this can be found in \cite{Ja1}, \cite{JP2}, \cite{Min}, \cite[6.3.1]{St1}.
Using this we extend the reproducing property given in \eqref{Reproduce} to $\mathcal B^1_\alpha$.

\begin{lemma}\label{LRepB1}
Let $\alpha>-1$.
For all $f\in\mathcal B^1_\alpha$ and $x\in\mathbb B$,
$$f(x)=\int_\mathbb B\mathcal R_\alpha(x,y)f(y)\,d\nu_\alpha(y).$$
\end{lemma}

\begin{proof}
Fix $x\in\mathbb B$ and let $f$ has the expansion \eqref{seriesH}.
We integrate in polar coordinates as in \eqref{Polar}.
The integrability condition is satisfied since for fixed $x\in\mathbb B$,
$\mathcal R_\alpha(x,y)$ is bounded by Theorem \ref{TEstBergman} and \eqref{bracketEst}.
Thus
$$\int_\mathbb B\mathcal R_\alpha(x,y)f(y)\,d\nu_\alpha(y)
=\frac{1}{V_\alpha}\,n\int_0^1r^{n-1}(1-r^2)^\alpha
\int_\mathbb S\mathcal R_\alpha(x,r\zeta)f(r\zeta)\,d\sigma(\zeta)\,dr.$$
The series expansions of $\mathcal R_\alpha(x,\cdot)$ in \eqref{RepKer}
and of $f$ in \eqref{seriesH} uniformly converge on $r\mathbb S$.
Writing these, changing the orders of the integral and the series, and then
using the orthogonality in \eqref{OrthSpher} and reproducing property
of $Z_m$ in \eqref{ReproduceZm} shows that
$$\int_\mathbb B\mathcal R_\alpha(x,y)f(y) d\nu_\alpha(y)
=\frac{n}{V_\alpha}\int_0^1r^{n-1}(1-r^2)^\alpha
\sum_{m=0}^\infty c_m(\alpha)r^{2m} S_m(|x|)S_m^2(r)p_m(x)\, dr.$$
We claim that the above series uniformly converges on $0\leq r\leq 1$.
To see this pick $\rho>1$ such that $\rho x\in\mathbb B$.
Then by the absolute convergence of the series in \eqref{seriesH} and the fact that
$S_m\geq 1$ by Lemma \ref{LSmrEst},
$$\sum_{m=0}^\infty|p_m(\rho x)|
\leq\sum_{m=0}^\infty S_m(\rho|x|)\,|p_m(\rho x)|<\infty,$$
and so there exists $C>0$ such that $|p_m(\rho x)|\leq C$ for all $m$.
Thus by homogeneity
$$|p_m(x)|=\frac{|p_m(\rho x)|}{\rho^m}\leq\frac{C}{\rho^m},$$
which with \eqref{gammamAsym} and Lemma \ref{LSmrEst} implies the uniform convergence.
Changing the orders of the integral and the series gives the desired result.
\end{proof}

\begin{lemma}\label{TestInt}
Let $\alpha>-1$ and $\delta>-1$.
The integral
$$\int_\mathbb B\mathcal R_\alpha(x,y)(1-|y|^2)^\delta\,d\nu(y)$$
is constant for all $x\in\mathbb B$.
\end{lemma}

\begin{proof}
For fixed $x\in\mathbb B$, the integrability condition is satisfied again
and we integrate in polar coordinates to obtain
$$\int_\mathbb B\mathcal R_\alpha(x,y)(1-|y|^2)^\delta\,d\nu(y)
=n\int_0^1r^{n-1}(1-r^2)^\delta
\int_\mathbb S\mathcal R_\alpha(x,r\zeta)\,d\sigma(\zeta)\,dr.$$
By the mean-value property of $\mathcal H$-harmonic functions,
the inner integral over $\mathbb S$ equals $\mathcal R_\alpha(x,0)=1$.
So the integral is $n\int_0^1r^{n-1}(1-r^2)^\delta\,dr=C$
for every $x\in\mathbb B$.
\end{proof}

\begin{proof}[Proof of Theorem \ref{Tprojection}]
Let $q$ be the conjugate exponent of $p$.
Before starting the proof we mention that the reason we used
the normalizing constant $V_\alpha$ is to normalize also the reproducing
kernel $\mathcal R_\alpha$, so that the series in \eqref{RepKer} starts
with $1$ and we have $\mathcal R_\alpha(x,0)=\mathcal R_\alpha(0,y)=1$.
This constant will show up in the formulas below, however it has no effect on
the boundedness of $P_\beta$ and can be ignored.

We first show the sufficiency of the condition \eqref{Projcon}.
We first note that if \eqref{Projcon} holds and $f\in L^p_\alpha$, then
$f\in L^1_\beta$.
This is clear when $p=1$, since in this case \eqref{Projcon} implies
$\beta>\alpha$.
When $1<p<\infty$, by H\"older's inequality
$$\int_\mathbb B|f(y)|(1-|y|^2)^\beta d\nu(y)
\leq \Bigl(\int_\mathbb B|f(y)|^p(1-|y|^2)^\alpha d\nu(y)\Bigr)^{\tfrac{1}{p}}
\Bigl(\int_\mathbb B (1-|y|^2)^{q(\beta-\tfrac{\alpha}{p})}d\nu(y)\Bigr)^{\tfrac{1}{q}}$$
and the last integral is finite since $q(\beta-\tfrac{\alpha}{p})>-1$ by \eqref{Projcon}.
So, $P_\beta f$ is a well-defined function on $\mathbb B$.
Also, since $\mathcal R_\alpha(\cdot,y)$ is $\mathcal H$-harmonic,
$P_\beta f$ is $\mathcal H$-harmonic.
Therefore, by Theorem \ref{TEstBergman}, it suffices to show that the operator $Q_\beta$
defined by
$$Q_\beta f(x):=\int_\mathbb B \frac{1}{[x,y]^{\beta+n}}\,f(y)(1-|y|^2)^\beta\,d\nu(y)$$
is bounded from $L^p_\alpha$ to $L^p_\alpha$.

If $p=1$, then by Fubini's theorem and Lemma \ref{LIntBracket} (with symmetry of $[x,y]$),
\begin{align*}
\|Q_\beta f\|_{1,\alpha}
&=\int_\mathbb B|Q_\beta f(x)|d\nu_\alpha(x)
=\frac{1}{V_\alpha}\int_\mathbb B|f(y)|(1-|y|^2)^\beta
\int_\mathbb B\frac{(1-|x|^2)^\alpha}{[x,y]^{\beta+n}}\,d\nu(x)\,d\nu(y)\\
&\lesssim\int_\mathbb B|f(y)|(1-|y|^2)^\beta\frac{1}{(1-|y|^2)^{\beta-\alpha}}d\nu(y)
=\|f\|_{1,\alpha},
\end{align*}
since $\beta-\alpha>0$ by \eqref{Projcon}.

If $1<p<\infty$, we apply Schur's test (see \cite[Theorem 1.8]{HKZ}).
According to this test, since
$$Q_\beta f(x)
=V_\alpha \int_\mathbb B \frac{1-|y|^2)^{\beta-\alpha}}{[x,y]^{\beta+n}}\,
f(y)\,d\nu_\alpha(y),$$
if we can find a positive function $h$ on $\mathbb B$ such that
\begin{equation}\label{h1}
\int_\mathbb B\frac{(1-|y|^2)^\beta}{[x,y]^{\beta+n}}h^q(y)\,d\nu(y)\lesssim h^q(x)
\qquad (x\in\mathbb B)
\end{equation}
and
\begin{equation}\label{h2}
(1-|y|^2)^{\beta-\alpha}\int_\mathbb B\frac{(1-|x|^2)^\alpha}{[x,y]^{\beta+n}}h^p(x)\,d\nu(x)
\lesssim h^p(y) \qquad (y\in\mathbb B),
\end{equation}
then $Q_\beta:L^p_\alpha\to L^p_\alpha$ is bounded.
We take $h(x)=(1-|x|^2)^{-\tfrac{\alpha+1}{pq}}$.
Then \eqref{h1} holds because by Lemma \ref{LIntBracket},
$$\int_\mathbb B\frac{(1-|y|^2)^{\beta-\tfrac{\alpha+1}{p}}}{[x,y]^{\beta+n}}\,d\nu(y)
\lesssim\frac{1}{(1-|x|^2)^{\tfrac{\alpha+1}{p}}}=h^q(x),$$
since $\beta-\tfrac{\alpha+1}{p}>-1$ by \eqref{Projcon}, and $\tfrac{\alpha+1}{p}>0$.
To see \eqref{h2} we again apply Lemma \ref{LIntBracket} to obtain
$$(1-|y|^2)^{\beta-\alpha}\int_\mathbb B
\frac{(1-|x|^2)^{\alpha-\tfrac{\alpha+1}{q}}}{[x,y]^{\beta+n}}\,d\nu(x)
\lesssim \frac{(1-|y|^2)^{\beta-\alpha}}{(1-|y|^2)^{\beta-\alpha+\tfrac{\alpha+1}{q}}}
=h^p(y).$$
since $\alpha-\tfrac{\alpha+1}{q}>-1$, and $\beta-\alpha+\tfrac{\alpha+1}{q}>0$ by \eqref{Projcon}.
This proves the sufficiency of \eqref{Projcon}.

To show the necessity of \eqref{Projcon}, suppose $P_\beta:L^p_\alpha\to L^p_\alpha$
is bounded.
Then the adjoint operator $P_\beta^*:L^q_\alpha\to L^q_\alpha$ is bounded (as usual we identify
the dual of $L^p_\alpha$ with $L^q_\alpha$), where $P_\beta^*$ is given by
$$P_\beta^*f(x)
=\frac{V_\beta}{V_\alpha}(1-|x|^2)^{\beta-\alpha}
\int_\mathbb B\mathcal R_\beta(x,y)f(y)\,d\nu_\alpha(y).$$
If $1<p<\infty$, then take $f=1$.
Since $f\in L^q_\alpha$, we must have $P_\beta^*f\in L^q_\alpha$.
But by Lemma \ref{TestInt}, we have $P_\beta^* f(x)=C(1-|x|^2)^{\beta-\alpha}$ and this belongs
to $L^q_\alpha$ only if $q(\beta-\alpha)+\alpha>-1$ which is same as \eqref{Projcon}.

If $p=1$, then $P_\beta^*$ is bounded on $L^\infty$ and we need to show that $\beta>\alpha$.
Taking $f=1$ again, we have $P_\beta^* f(x)=C(1-|x|^2)^{\beta-\alpha}$ which belongs to
$L^\infty$ only if $\beta\geq\alpha$.
What remains is to show that $\beta=\alpha$ can not be true.
Assume now that $\beta=\alpha$ and for $x_0\in\mathbb B$, define $f_{x_0}$ on $\mathbb B$ by
$$f_{x_0}(y)=
\begin{cases}
\dfrac{|\mathcal R_\beta(x_0,y)|}{\mathcal R_\beta(x_0,y)},&\text{if $\mathcal R_\beta(x_0,y)\neq 0$};\\
1,&\text{if $\mathcal R_\beta(x_0,y)=0$}.
\end{cases}$$
Clearly $\|f_{x_0}\|_\infty=1$, and by Theorem \ref{TIntegral},
$$P_\beta^*f_{x_0}(x_0)= \int_\mathbb B|\mathcal R_\beta(x_0,y)|\,d\nu_\beta(y)
\sim 1+\log\frac{1}{1-|x_0|^2}.$$
Because $P_\beta^*f_{x_0}$ is continuous, this implies
$\|P_\beta^*f_{x_0}\|_\infty\gtrsim 1+\log\dfrac{1}{1-|x_0|^2}\to \infty$ as
$|x_0|\to 1^-$.
This contradicts to the boundedness of $P_\beta^*$ on $L^\infty$.
Thus $\beta=\alpha$ can not be true and the necessity of \eqref{Projcon} is proved.

Finally, suppose \eqref{Projcon} holds and $f\in\mathcal B^p_\alpha$.
Then $f\in\mathcal B^1_\beta$ as shown in the beginning of the proof,
and it follows from Lemma \ref{LRepB1} that $P_\beta f=f$.
\end{proof}

We finish by proving Theorem \ref{TProjBloch}.

\begin{proof}[Proof of Theorem \ref{TProjBloch}]
Let $f\in L^\infty$.
First, since $\mathcal R_\alpha(0,y)=1$,
$$\bigl|P_\alpha f(0)\bigr|
=\Bigl|\int_\mathbb B\mathcal R_\alpha(0,y)f(y)\,d\nu_\alpha(y)\Bigr|
\leq\|f\|_\infty.$$
Next, by Theorem \ref{TPartial} and Lemma \ref{LIntBracket}, for $i=1,\dots,n$,
\begin{align*}
\biggl|\frac{\partial}{\partial x_i}\,P_\alpha f(x)\biggr|
&=\biggl|\int_\mathbb B\frac{\partial}{\partial x_i}\mathcal R_\alpha(x,y)f(y)\,d\nu_\alpha(y)\biggr|
\lesssim \|f\|_\infty\int_\mathbb B\frac{(1-|y|^2)^\alpha}{[x,y]^{\alpha+n+1}}\,d\nu(y)\\
&\lesssim \frac{\|f\|_\infty}{1-|x|^2}.
\end{align*}
Hence $(1-|x|^2)\bigl|\nabla P_\alpha f(x)\bigr|\lesssim\|f\|_\infty$ and
combining above we conclude that $P_\alpha f\in\mathcal B$ with
$\|P_\alpha f\|_\mathcal B\lesssim \|f\|_\infty$.
\end{proof}

\section*{Acknowledgment}
The author thanks H. Turgay Kaptano\u{g}lu of Bilkent University for helpful discussions
and bringing \cite{Pen} to his attention.
He also thanks Salvador P\'{e}rez-Esteva of Universidad Nacional Aut\'{o}noma de M\'{e}xico for providing
information about \cite{Pen}.

\bibliographystyle{amsalpha}

\end{document}